\documentclass[10pt,reqno]{amsart}

\usepackage[a4paper,asymmetric]{geometry}
\usepackage{latexsym, amsmath, amssymb, esint}
\usepackage{bm}
\usepackage{graphics,epsfig,graphicx,graphics}
\usepackage{color, xcolor}
\usepackage{tikz}
\usepackage[hyperpageref]{backref}
\usetikzlibrary{patterns}
\usepackage[english]{babel}
\usepackage[colorlinks=true,urlcolor=blue, citecolor=red,linkcolor=blue,
linktocpage,pdfpagelabels, bookmarksnumbered,bookmarksopen]{hyperref}
\newtheorem{theorem}{Theorem}[section]
\newtheorem{cor}[theorem]{Corollary}
\newtheorem{lemma}[theorem]{Lemma}
\newtheorem{prop}[theorem]{Proposition}

\theoremstyle{definition}

\newcommand{\ep}{\varepsilon}
\newcommand{\RR}{\mathbb{R}}
\newcommand{\rn}{\RR^N}

\newcommand{\bd}{\partial}

\newcommand{\bho}{B_{H_0}}
\newcommand{\bhor}[1]{B_{H_0}({#1})}
\newcommand{\bhorx}[2]{B_{H_0}({#2},{#1})}

\newcommand{\diver}{{\rm div}}

\newcommand{\haus}{\mathcal{H}^{N-1}}
\newcommand{\Huno}{{\sf M}}
\newcommand{\hnorm}{{\mathcal{I}_p}}

\newcommand{\lam}{\lambda}
\newcommand{\la}{\lambda}

\newcommand{\nxi}{\nabla_\xi}

\newcommand{\Om}{\Omega}

\newcommand{\pa}{\partial}
\newcommand{\pCap}{{\sf{Cap}}_p^H}

\newcommand{\pesp}{\frac{p-N}{p-1}}

\newcommand{\sdueij}[1]{S^2_{ij}(#1)}

\newcommand{\tr}{{\sf{Tr}}}

\numberwithin{equation}{section}

\begin{document}
\title[Wulff shape characterizations]{Wulff shape characterizations in overdetermined anisotropic elliptic problems}

\author[C. Bianchini]{Chiara Bianchini}
\author[G. Ciraolo]{Giulio Ciraolo}

\address{C. Bianchini, Dipartimento di Matematica e Informatica ``U. Dini'', Universit\`a degli Studi di Firenze, Viale Morgagni 67/A, 50134 Firenze - Italy}
\email{cbianchini@math.unifi.it}

\address{G. Ciraolo,  Dip.to di Matematica e Informatica, Universit\`a degli Studi di Palermo, Via Archirafi 34, 90123, Palermo - Italy}
\email{giulio.ciraolo@unipa.it}

\maketitle

\begin{abstract}
We study some overdetermined problems for possibly anisotropic degenerate elliptic PDEs, including the well-known Serrin's overdetermined problem, and we prove the corresponding Wulff shape characterizations by using some integral identities and just one pointwise inequality. Our techniques provide a somehow unified approach to this variety of problems.
\end{abstract}




\section{Introduction}

The aim of this paper is to characterize the shape of a domain in terms of solutions to overdetermined elliptic problems.
In this kind of problems ``too many'' conditions are prescribed at the boundary of the domain and hence, typically, they are not well-posed and the existence of a solution imposes strong restrictions on the shape of the domain where the problem is defined. 

The equations that we consider in this paper arise from the study of some variational problem in a possibly anisotropic medium and they are of elliptic type, where the ellipticity may be singular or degenerate. Apart from its mathematical interest, variational problems in anisotropic media naturally arise in the study of crystals and whenever the microscopic environment of the interface of a medium is different from the one in the bulk of the substance so that anisotropic surface energies have to be considered, as well as in noise-removal procedures in digital image processing, crystalline mean curvature flows and crystalline fracture theory (see \cite{BNP,BP,Ch,CFV,CFV2,NP,OBGXY,Tay,TCH,Wu} and references therein).

The study of overdetermined problems started with the seminal paper of Serrin \cite{Se}, where it is proved that if there exists a solution to
\begin{equation} \label{PbSerrin}
\begin{cases}
\Delta u = -1 & \textmd{in } \Omega \,, \\
u=0 & \textmd{on } \partial \Omega \,, \\
u_\nu=c & \textmd{on } \partial \Omega \,,
\end{cases}
\end{equation}
for some constant $c$ and some bounded domain $\Omega\subset \rn$, then $\Omega$ must be a ball and $u$ is radially symmetric. Here, $u_\nu$ denotes the \emph{inward} normal derivative to $\Omega$. Immediately after \cite{Se}, Weinberger \cite{We} provided a simplified proof of Serrin's result by using some integral identities. These two papers originated two different branches of investigations for symmetry results in overdetermined problems.

Indeed, in \cite{Se} Serrin introduced the PDE's community to the method of moving planes, which was firstly used by Alexandrov in \cite{Al} to prove the well-known \emph{Alexandrov's soap bubble theorem}. The method of moving planes is very flexible and can be used to prove symmetry results for much more general uniformly elliptic equations (see \cite{Se}).
It has been employed also in other types of problems (\cite{CMS,GNN,Re}) and for overdetermined problems in rotationally symmetric spaces different from the Euclidean space (see \cite{KP, Molz}).
However, the method of moving planes does not work (at least applied in a standard way) in manifolds which are not rotationally symmetric such as Finsler manifolds (see below for a more detailed discussion). 

Weinberger's approach has also been generalized in several direction. The main idea of this approach is to use some integral identities and a maximum principle for the so-called P-function.
This approach was refined in \cite{GL} and more recently in \cite{FGK} and \cite{FK} where the symmetry result was proved for a large class of quasilinear equations. We mention that the use of the $P$-function has been employed also for anisotropic spaces in \cite{WX} under quite restrictive assumptions on the regularity of the norm which describes the anisotropy. 

Starting from Weinberger's approach, in \cite{BNST} the authors gave another proof of Serrin's symmetry result, still by using integral identities but not invoking the use of the $P$-function and maximum principle, thus weakening the required regularity on the solution $u$. Indeed, by using some integral identity and just one basic pointwise inequality (Cauchy-Schwarz) on the Hessian of the solution $D^2u$, the authors prove that $D^2u$ is a multiple of the identity matrix, which easily leads to the conclusion. This strategy has been also used in \cite{CS} to extend Serrin's result to the Finsler Laplacian, in \cite{BCS} for the exterior Serrin's problem in anisotropic spaces (see below for a more detalied description) and in \cite{CV} for an overdetermined problem on the round sphere.

In this paper we refine the approach introduced in \cite{BNST} and \cite{CS}.  
More precisely we recover the Wulff shape of the domain by using Alexandrov's theorem and we provide a sort of general scheme which can be applied to several problems. This approach is new even in the Euclidean case, and it allows us to generalize the results in \cite{BCS} \cite{BNST} and \cite{CS} to degenerate operators in an anisotropic setting. In particular we will give symmetry results for interior and exterior overdetermined problems for the $p$-Laplace operator in the Finsler setting (which clearly includes the Euclidean case).

In order to make the statements more clear, we introduce some notation. Given a norm $H$, we say that a set $\Omega$ is Wulff shape of $H$ if $\Omega$ is a level set of the dual norm $H_0$ (see Section \ref{section_preliminaries} below); up to translations, in this case we write $\Omega=\bhor{r}$, where $\bhor{r}=\{x\ :\ H_0(x)<r\}$. In the case $r=1$ we omit the dependecy on $r$, i.e. $\bho=\{x\ :\ H_0(x)<1\}$.

Given a function $H:\rn \to [0,+\infty)$, we define
\begin{equation}\label{V_def}
V(\xi)=\frac 1p\, H^p(\xi) \,,  \quad \xi \in \rn \,.
\end{equation}
We will consider the case when $H$ is a norm in the class
\begin{equation} \label{Ip_def}
\hnorm=\{ H\in C^{2,\alpha} (\rn\setminus\{0\}) \,, V\in C^2_+(\rn\setminus\{0\}) \ \},
\end{equation}
with $p>1$ and for some $\alpha \in (0,1)$, that is a regular norm $H$ where $H^p$ is a twice continuously differentiable function in $\rn\setminus\{0\}$ whose Hessian matrix has positive eigenvalues uniformly bounded away from $0$ (or equivalently with $\bho$ uniformly convex, see discussion in Section \ref{section_norms}).


Our first main result regards a generalization of Problem \eqref{PbSerrin}. More precisely, we consider the minimization problem
\begin{equation} \label{Pb_min_Serrin}
\min_{W_0^{1,p}(\Omega)} \int_\Omega \left(\frac{1}{p} H(\nabla u)^p - u \right) dx   \,,
\end{equation}
where $\Omega \subset \RR^N$ is a bounded domain. It is well known that, if $H \in \hnorm$ then \eqref{Pb_min_Serrin} has a unique solution and the minimizer $u\in W_0^{1,p}(\Omega)$ of \eqref{Pb_min_Serrin} is a weak solution of the Dirichlet problem
\begin{equation}\label{int-pbu}
\begin{cases}
\Delta^H_p u =-1\qquad&\text{in }\Omega \,,\\
u=0\qquad&\text{on }\bd\Omega \,.
\end{cases}
\end{equation}
Here, $\Delta_p^H$ is the \emph{Finsler $p$-Laplacian} (or \emph{anisotropic $p$-Laplacian}) operator which is given by
\begin{equation}\label{pLapl}
\Delta^H_p u= \diver(H^{p-1}(Du)\nabla_{\xi}H(Du)) 
\end{equation}
in the sense of distributions; more precisely, \eqref{int-pbu} reads as
\begin{equation} \label{int-pbu_distr}
\int_\Omega H^{p-1}(Du)\,\langle \nabla_{\xi}H(Du); D \phi \rangle \,dx = \int_\Omega \phi \, dx \,,
\end{equation}
for any $\phi \in C_0^1(\Omega)$.

By a straightforward computation, it is easy to show that if $\Omega= \bhor{r}$ then the solution to \eqref{int-pbu} is given by
\begin{equation}\label{u=}
u(x) =\frac{(p-1)\big(r^{\frac p{p-1}}-H_0^{\frac p{p-1}}(x)\big)}{p N^{\frac 1{p-1}}},\qquad x\in B_{H_0}(r).
\end{equation}
In particular $H(Du)$ is constant on {$\partial \bhor{r}$}. In our first main result we show that the reverse assertion still holds, that is we provide a characterization of the Wulff shape in terms of the solution to \eqref{int-pbu}.

\begin{theorem} \label{thm_serrin_int}
Let $\Omega \subset \RR^N$ be a bounded domain with $\partial \Omega \in C^{2,\alpha}$. Let $H$ be a norm in $\hnorm$, with $p>1$.

If there exists a solution $u$ to \eqref{int-pbu} satisfying
\begin{equation}\label{HDuConst-int}
H(Du)=C \qquad\text{on }\bd\Omega \,
\end{equation}
for some $C>0$, then $\Omega$ is Wulff shape, that is (up to translations) there exists $r>0$ such that $\Omega=\bhor{r}$ and $u$ is given by \eqref{u=}.
\end{theorem}
As already mentioned, the proof of Theorem \ref{thm_serrin_int} consists of some integral identities, Cauchy-Schwarz inequality, and it is completed by using Alexandrov's Theorem.
Notice that this approach is new even in the Euclidean case for the usual $p-$Laplace operator for $p\neq 2$. For $p=2$, Theorem \ref{thm_serrin_int} has been proved in \cite{CS} under the weaker assumption that the boundary of $\Omega$ is of class $C^1$ (see also \cite{WX} for $H \in C^4(\rn \setminus \{ 0 \})$). The case $p \neq 2$ has been investigated in \cite{LLZ} for $H \in C^4(\rn \setminus \{ 0 \})$ following the ideas in \cite{FK,FGK,WX}; however, in \cite{LLZ} it is not clear how the approximation argument used for the $P$-function in \cite{LLZ} can exclude that the $P$-function attains the maximum at critical points of $u$.
\par
Now we describe the results regarding exterior domains. Here, we are motivated by the study of overdermined problems for the anisotropic $p$-capacity, which is defined by
\begin{equation}\label{pcap}
\pCap(\Omega)=\inf \left\{  \frac 1p\int_{\rn} H^p(D\varphi)\, dx,\ \varphi\in C^{\infty}_0(\rn), \varphi(x)\ge 1 \text{ for }x\in\Omega \right\} \,,
\end{equation}
for $N\geq 3$ and $1<p<N$.
If $H$ is a norm in the class $\hnorm$, the integral operator is strictly convex and \eqref{pcap} admits a unique solution $u$, which satisfies
\begin{equation}\label{ext-pbu}
\begin{cases}
\Delta^H_p u =0\qquad&\text{in }\rn\setminus\overline{\Omega},\\
u=1\qquad&\text{on }\bd\Omega \,,\\
u\to 0\qquad&\text{as } H(x)\to +\infty \,.
\end{cases}
\end{equation}

We mention that, for $p=2$, the Euclidean capacity of a set $\Omega$ measures the capacitance of the set, that is the total charge $\Omega$ can hold while maintaining a given potential energy, with respect to an idealized ground at infinity. 
Analogously ${\sf{Cap}}_2^H(\Omega)$  measures the \emph{anisotropic capacitance} of $\Omega$, that is the total charge the set $\Omega$ can hold  while embedded in an {anisotropic dielectric medium} and maintaining a given potential energy, with respect to an idealized ground at infinity.

The case $p \neq 2$ is also interesting from the physical point of view. Indeed, there are many physical phenomena where the background medium is described by a nonlinear law of the form ${ \bf J} = |{\bf E}|^{p-2} {\bf E}$, $p>1$. For instance, in deformation theory of plasticity ${ \bf E}$ and ${ \bf J} $ represent the infinitesimal strain and stress  \cite{PCS, Suq}, respectively, for nonlinear dielectrics 
problems ${ \bf E}$ and ${ \bf J} $ are the electric field and current \cite{Ein, GNP, LKohn}, respectively, and this types of laws arise also in fluid flows where ${ \bf E}$ and ${ \bf J} $ are the rate of stress and fluid strain \cite{AR, Ruz}.

Symmetry results for problems involving (Euclidean) capacity go back to P\'olya and Szeg\"o and subsequents authors (see for instance \cite{Sz,PS,Mo,Ke,Xiao} and references therein). In \cite{Re} the author used the method of moving planes to analyze Problem (\ref{ext-pbu}) in the Euclidean case and he proved that if there exists a solution of \eqref{ext-pbu} satisfying $|Du|=C$ on $\partial \Omega$, then $\Omega$ must be an Euclidean ball. In the anisotropic setting, the method of moving planes is no more applicable and the result was extended to the anisotropic Laplacian (i.e. for $p=2$) in \cite{BCS}. In this paper, we prove the further generalization to any $1<p<N$ in the anisotropic setting.

\begin{theorem} \label{thm_serrin_ext}
Let $\Omega$ be a bounded convex domain of $\rn$ with boundary of class $C^{2,\alpha}$ and let $H$ be a norm in $\hnorm$, $1<p<N$. If there exists a solution $u$ to \eqref{ext-pbu} such that
\begin{equation} \label{HDuConst-ext}
H(Du)=C \quad \textmd{on } \partial \Omega \,,
\end{equation}
then $\Omega$ is Wulff shape, that is (up to translations) there exists $R>0$ such that $\Omega=\bhor{R}$ and $u$ is given by
\begin{equation} \label{uexplext}
u(x)= \Big(\frac{H_0(x)}{R}\Big)^\frac{p-N}{p-1} \,.
\end{equation}
\end{theorem}
For $p=2$, a physical interpretation of Theorem \ref{thm_serrin_ext} is the following: the Wulff shape is the unique shape a conductor can have  if, while embedded in an anisotropic dielectric and maintaining a given potential energy (with respect to an idealized ground at infinity), the intensity of the electrostatic field is constant on its boundary.

Notice that the reverse assertion is also true. More precisely, if $\Omega$ is Wulff shape then a straightforward computation shows that $u$ is given by \eqref{uexplext} and $H(Du)$ is constant on $\bd\Omega$. Hence Theorem \ref{thm_serrin_ext} gives a complete characterization of the Wulff shape for Problem (\ref{ext-pbu})-(\ref{HDuConst-ext}).

As a byproduct of the technique used for proving Theorem \ref{thm_serrin_ext}, we can tackle another overdetermined problem in exterior domains for the $p$-capacity which was recently considered in \cite{AM} in the Euclidean case and for $p=2$ (see also \cite{BiCiPali}). In the following, $\Huno_{\partial \Omega}$ denotes the anisotropic mean curvature of $\partial \Omega$ (see \eqref{MHOm} below for its definition). 

\begin{theorem} \label{thm_agomaz}
Let $\Omega$ be a bounded convex domain of $\rn$ of class $C^{2,\alpha}$; let $H$ be a norm in $\hnorm$, $1<p<N$. If there exists a solution $u$ to (\ref{ext-pbu}) such that
	\begin{equation}\label{agomazz-u}
		\int_{\bd\Omega}H(\nu)H^{2(p-1)}(Du)\left(\frac{\Huno_{\partial \Omega}}{N-1}- \frac{p-1}{p-N}\frac{H(Du)}{u}\right)\le 0,
	\end{equation}
	then $\Omega$ is Wulff shape.
\end{theorem}
In the Euclidean case for $p=2$, Theorem \ref{thm_agomaz} was proved in \cite{AM} by using a conformal mapping method. Here, we generalize the approach in \cite{BiCiPali} to the Finsler setting and for the $p-$capacity problem.
Notice that, even in the Euclidean case, the method of moving planes does not seem to be suitable for proving Theorem \ref{thm_agomaz} since the overdetermined condition \eqref{agomazz-u} is of nonlocal type.

As for Theorem \ref{thm_serrin_int}, the proofs of Theorems \ref{thm_serrin_ext} and \ref{thm_agomaz} are based on an integral identity involving the second elementary symmetric function of a matrix $W$ related to the Hessian of a suitable power of the solution, combined with Cauchy-Schwarz inequality.
In particular, the idea is to show that the overdetermined conditions \eqref{HDuConst-ext} and \eqref{agomazz-u} are in fact equivalent to ask that such a matrix $W$ attains the equality sign in Cauchy-Schwarz inequality. This gives a strong condition on the structure of $u$ which implies that the anisotropic mean curvature of $\partial \Omega$ is constant, that is $\Omega$ is Wulff shape. We clarify our approach in the following subsection.

\subsection{{Framework of the proofs}} \label{subsect_ideas}

We briefly describe the framework of the proofs of Theorems \ref{thm_serrin_int}, \ref{thm_serrin_ext} and \ref{thm_agomaz}.
All these proofs make the use of a common approach which we explain in the following.

Let $\Omega \subset \rn$ be a bounded domain and let $u$ be the solution to
\begin{equation} \label{general_eq}
\begin{cases}
\Delta_p^H u = c_1 & \textmd{in } D \,, \\
u=c_2 & \textmd{on } \partial D\,,
\end{cases}
\end{equation}
where either $D= \Omega$ or $D= \rn\setminus\overline{\Omega}$, and $c_1$ and $c_2$ are constants (if $D=\rn\setminus\overline{\Omega}$ a condition at infinity is also given). Assume to know that
%
%
$$
\textit{there exists } m \in \RR \ \textit{ such that if } \  \Omega =   \bho \ \textit{ then }  \ u^m (x)=
a + b H_0^{\frac{p}{p-1}} (x) \textit{ for some } a,b \in \RR
$$
(notice that we will set $m=1$ for \eqref{int-pbu} and $m=p/(p-N)$ for \eqref{ext-pbu}).

For $w=u^m$ we define $W=\nxi^2 V(Dw) D^2w$, where $V(\cdot)=H^p(\cdot)/p$, so that
$$
\tr(W)=\Delta_p^H w \,. 
$$
From \eqref{general_eq} we find that $w$ solves an equation of the form
$$
\tr(W)= b(w,H(\nabla w))
$$
for some function $b$. Then we apply the following scheme.

\begin{itemize}
\item[\emph{Step 1}.] Starting from a differential identity for the second elementary symmetric function $S_2$ applied to Hessian matrices (see Lemma \ref{lemmaS2W=}), we derive an integral identity for $W$ involving $\tr(W)$ (Lemmas \ref{lemma_ident_int} and \ref{lemma_ident_ext} for the interior and exterior problems, respectively).

\noindent From Cauchy-Schwarz inequality we have that $(N-1)(\tr W)^2 \geq 2N S_2(W)$ and we obtain an integral inequality for $w$ where the equality sign is attained if and only if the matrix $W$ is a multiple of the identity matrix (see Corollaries \ref{int-condiz-int} and \ref{ext-condiz-int}).
\item[\emph{Step 2}.] By using an additional constraint (the overdetermining condition in the original problem for $u$) we prove that the equality sign  holds in the inequality obtained at \emph{Step 1}, and hence $W=\lambda Id$ for some constant $\lambda$.
Since $W=\nxi^2 V(Dw) D^2w$, we have that $D^2w= \lambda (\nxi^2 V(Dw))^{-1}$ (Lemmas \ref{lemmaOmegaWulff-int} and \ref{lemmaOmegaWulff-ext} for the interior and exterior problems, respectively).
\item[\emph{Step 3}.] With all these ingredients at hand we can write the anisotropic mean curvature $\Huno_{\partial \Omega}$ of the boundary of $\Omega$ in terms of $H(Dw)$ and conclude that
$\Huno_{\partial \Omega}$ is constant, which implies the desired Wulff shape characterizations by Alexandrov's theorem.
\end{itemize}


We mention that \emph{Step 1} presents some technical difficulties especially for the interior problem. Indeed, being $u$ and $V$ of class $C^{1,\alpha}$ and $C^{2,\alpha}$ respectively, we achieve \emph{Step 1} by using a careful approximation argument, since more regularity is needed in order to write the pointwise differential identity in Lemma \ref{lemmaS2W=}.

We notice that for Theorems \ref{thm_serrin_int} and \ref{thm_serrin_ext} the overdetermined condition mentioned at \emph{Step 2} is the requirement that $H(Du)$ is constant on $\partial \Omega$. In Theorem \ref{thm_agomaz}, the additional constraint is the condition \eqref{agomazz-u}. 

Let us highlight that in \emph{Step 2} we prove something more. Indeed, in Lemma \ref{lemmaOmegaWulff-int} and Lemma \ref{lemmaOmegaWulff-ext} we show that assuming $H(Du)$ constant on $\bd\Omega$ is equivalent to impose that the equality sign is attained in Cauchy-Schwarz inequality.

We conclude this introduction by noticing that the framework described above is inspired from \cite{CS}, where the authors prove the Wulff shape characterization for the interior problem \eqref{int-pbu} and \eqref{HDuConst-int} when $p=2$. 
However, our approach is not a straightforward generalization of the one in \cite{CS}. Indeed, we study more general (and degenerate) equations, which introduce several technical difficulties especially regarding the regularity of the solutions. 
Moreover, also the general scheme of the proof differs: indeed in \cite{CS} the authors are able to conclude at Step 2, since the equality case in Cauchy-Schwarz inequality implies that $D^2w$ is a multiple of the identity matrix in $\Omega$. This doesn't happen when $p \neq 2$. However, the equality case gives us a strong information on $D^2w$ which, when used on the boundary of $\Omega$, implies that the mean curvature of $\partial \Omega$ is constant.

\subsection*{Organization of the paper}
The paper is organized as follows. In Section \ref{sect_preliminaries} we introduce some notation, recall basic facts on the norms in $\rn$ and describe some useful properties of the elementary symmetric function $S_2$. In Section \ref{sect_preliminaries} we give some preliminary result. At the beginning of Section \ref{sect_prelim_II} we prove a crucial differential identity which will be used in Section \ref{section4} to achieve Step 1 for interior and exterior problems. In Section \ref{section_final_proofs} we complete the proofs of our main theorems. Finally, Appendix \ref{appendix-valoreC} is devoted to compute the value of the constant $C$ which appears in the Neumann boundary constraint (\ref{HDuConst-ext}) and in Appendix \ref{appendix_lower_bound} we give a lower bound on the gradient of the solution of \eqref{ext-pbu}.

\subsection*{Acknowledgments} The authors are indebted to Lorenzo Brasco for the discussions they had together and addressing to \cite{AKM}. The authors warmly thank Andrea Cianchi, Nicola Fusco and Paolo Salani for their remarks.

The authors have been supported by Fir Project 2013 ``Geometrical and qualitative aspects of PDE's'' of MIUR (Italian Ministry of Education) and by the GNAMPA group of Indam.

\section{Miscellanea} \label{sect_preliminaries}

\subsection{{Notation}}

For a subset $\Omega$ of $\rn$ we denote by $|\Omega|$ its volume, 
so that:
$$
|\Omega|=\int_{\Omega}d\mathcal{H}^{N}(x). 
$$
Given a convex set $\Omega$, we denote by $\nu=(\nu^1,\ldots,\nu^N)$ its inner unit normal vector.
For $j \in \{1,\ldots,N\}$, $\nu_j=(\nu_j^1,\ldots,\nu_j^N)$ will indicate the vector of derivatives of $\nu$ with respect to the variable $x_j$.

Given a function $u:\Omega\to\rn$, the gradient $Du$ evaluated at $x \in \Omega$ is the element $Du(x)$ of the dual space of $\rn$. 
Unless otherwise stated, we will use the variable $x$ to denote a point in the ambient space $\RR^N$ and $\xi$ for an element in its dual space. The symbols $D$ and $\nxi$ will denote the gradients with respect to the $x$ and $\xi$ variables, respectively.

We notice that, unless otherwise specified, we adopt the Einstein summation convention.

\subsection{Norms of $\rn$} \label{section_norms}
We consider the space $\rn$ endowed with a generic norm $H: \rn \to \RR$ such that:
\begin{itemize}
	\item[(i)] $H$ is convex;
	\item[(ii)] $H(\xi) \geq 0$ for $\xi \in \rn$ and $H(\xi)=0$ if and only if $\xi=0$;
	\item[(iii)] $H(t\xi) = |t| H(\xi)$ for $\xi \in \rn$ and $t\in \RR$.
\end{itemize}

More precisely we identify the dual space of $\rn$ with $\rn$ itself via the scalar product $\langle \cdot;\cdot\rangle$.
Accordingly the space $\rn$ turns out to be endowed with the dual norm $H_0$ given by
\begin{equation}\label{defH0}
H_0(x)=\sup_{\xi\neq 0}\frac{\langle {x};{\xi}\rangle}{H(\xi)}\quad\text{ for
}x\in\rn\,,
\end{equation}
On the other hand we can define $H$ in terms of $H_0$ as
\begin{equation*}
H(\xi) = \sup_{x\neq 0} \frac{\langle x ; \xi\rangle}{H_0(x)},\quad \xi \in \rn.
\end{equation*}
Notice that $H$ results to be the \emph{support function} (see \cite{Sc}, Section 1.7 for the definition and details) of the unitary ball $B_{H_0}=\{x\in\rn\ :\ H_0(x)< 1\}$ of $H_0$ and, in turn, $H_0$ is the support function of $B_H=\{\xi\in\rn\ :\ H(\xi)< 1 \}$.
The two convex sets $B_{H_0}$ and $B_H$ are both centrally symmetric and they are polar of each other.

As already mentioned, we denote by $\bhor{r}$ the ball centered at $O$ with radius $r$ in the norm $H_0$, i.e.
\begin{equation*}
\bhor{r} = \{x \in \rn:\ H_0(x) < r\}.
\end{equation*}
Analogously, we define 
$$
B_{H}(r)=\{\xi \in \rn:\ H(\xi) < r\} \,.
$$
The sets $\bhor{r}$ (as well as their translations) are named \emph{Wulff shapes} of $H$ and they are in fact homothetic copies of the ball $\bho$. 

From \cite[Corollary 1.7.3]{Sc}, we have that
$H_0 \in C^1(\rn \setminus \{0\})$ if and only if $B_{H}$  is strictly convex.
Moreover, we notice that if $H\in C^2(\RR^N\setminus\{0\})$ and $B_{H}$ is uniformly convex, then the same holds for $H_0$ and $\bho$. As noticed in \cite[Appendix A]{CFV} if $H\in C^2(\RR^N\setminus\{0\})$ and $B_{H}$ is uniformly convex then $H \in \hnorm$. Viceversa if $H \in \hnorm$ then it is easy to prove that 
$$
H_{ij}(\xi) \zeta_i \zeta_j \geq \mu H(\xi)^{-1} |\zeta|^2 \quad \text{for any } \xi \in \rn\setminus\{0\}\,, \ \zeta \in \nabla H(\xi)^\perp \,,
$$
for some $\mu >0$, which implies that $B_H$ is uniformly convex (see again \cite[Appendix A]{CFV}).

Hence if we consider a norm $H\in\hnorm$, with $\hnorm$ given by \eqref{Ip_def} and $p>1$, the corresponding ball $B_H$ and dual ball $B_{H_0}$ are $C^{2,\alpha}$ uniformly convex and centrally symmetric sets. 

Since all norms in $\rn$ are equivalent, there exist positive constants $\sigma_1$ and $\sigma_2$ such that
\begin{equation} \label{norms equiv}
\sigma_1 |\xi| \leq H(\xi) \leq \sigma_2 |\xi|,\quad \xi \in \rn.
\end{equation}

Let $H \in C^1(\rn \setminus \{0\})$, from the homogeneity property (iii)  we have
\begin{equation}\label{Hhomog1}
\langle\nxi H(\xi);\xi\rangle = H(\xi),\quad \xi \in \rn,
\end{equation}
where the left hand side is taken to be $0$ when $\xi = 0$.
Moreover, if $H \in C^2(\rn \setminus \{0\})$, then
\begin{equation}\label{Hhomog2}
	\nxi^2H(\xi)\xi=0.
\end{equation}

The following properties hold provided $H \in C^1(\rn \setminus \{0\})$ and $\bho$ is strictly convex (see \cite[Section 3.1]{CS}):
\begin{equation}\label{HDHo=1}
H(D_\eta H_0(\eta))=1,\quad H_0(\nabla_\xi H(\xi))=1,
\end{equation}
for every $\xi,\eta \in \rn\setminus \{0\}$.
Furthermore, the map $H \nxi H $ is invertible with
\begin{equation}\label{HnablaH inverse}
H\nxi H = (H_0 \nxi H_0)^{-1}.
\end{equation}
From \eqref{HDHo=1} and the homogeneity of $H_0$, \eqref{HnablaH inverse} is equivalent to
\begin{equation}\label{HDHo}
H(\xi)\; D_\eta H_0(\nabla_\xi H(\xi)) = \xi \,.
\end{equation}
When $H$ and $H_0$ are of class $C^2(\RR^N\setminus\{0\})$, by differentiating this expression and using (\ref{Hhomog1}) and (\ref{Hhomog2}), we obtain
\begin{equation} \label{miracolo}
\nabla_\xi^2 U\; D_\eta^2 U_0 (\nabla_\xi H) = Id \,,
\end{equation}
for every $\xi\in\rn\setminus\{0\}$, where $U=H^2/2$ and $U_0=H_0^2/2$. More generally, we have the following lemma.

\begin{lemma} \label{lem_inversa_D2V} 
Let $H\in \hnorm$, $p>1$, and let $H_0$ be its dual norm. Set $V=H^p/p$. For $\xi \in \rn$ and $\xi \neq 0$, we have the following identity
\begin{eqnarray}\label{eq_inversa_D2V}
\qquad (\nxi^2 V(\xi))^{-1}_{ij}=\frac { H^{2-p}(\xi)}{p-1}   \left( \partial_{\eta_i} H_0 (\nxi H(\xi))\partial_{\eta_j}H_0(\nxi H(\xi))  + (p-1) \partial^2_{\eta_i \eta_j} H_0 (\nxi H(\xi)) \right).
%
\end{eqnarray}
\end{lemma}

\begin{proof}
The formula can be verified by direct calculations. We denote by $A$ the matrix whose entries are
$$
A_{ik}(\xi)= H^{2-p}(\xi)  \left( \partial_{\eta_i} H_0 (\nxi H(\xi))\partial_{\eta_j}H_0(\nxi H(\xi))  + (p-1) \partial^2_{\eta_i \eta_j} H_0 (\nxi H(\xi)) \right) 
$$
and we compute $A(\xi)\;\nxi^2 V(\xi)$:
\begin{multline}\label{ANablaV}
\Big(A(\xi)\;\nxi^2 V(\xi)\Big)_{ij}\\
=H^{2-p}(\xi) H^{p-2} \left( \partial_{\eta_i} H_0 \; \partial_{\eta_k} H_0 +(p-1) 
\partial^2_{\eta_i \eta_k} H_0 \right)\left( (p-1)H_{\xi_k}H_{\xi_j}+HH_{\xi_k\xi_j} \right),
\end{multline}
where $H_0$ is evaluated at $\nxi H(\xi)$ while $H$ is evaluated at $\xi$.
By recalling that homogeneity properties (\ref{Hhomog1}) and (\ref{Hhomog2}) hold for both $H$ and $H_0$, and by using (\ref{HDHo=1}) and (\ref{HDHo}) we have
\begin{eqnarray*}
\partial_{\eta_k} H_0 H_{\xi_k}=H_0(\nxi H)=1, \\
H\partial_{\eta_i} H_0 \; \partial_{\eta_k} H_0 \; H_{\xi_j\xi_k}=  \xi_k H_{\xi_j\xi_k} \partial_{\eta_i} H_0=0,\\
H_{\xi_k}H_{\xi_j}\partial^2_{\eta_i\eta_k} H_0 =0.
\end{eqnarray*}
Hence (\ref{ANablaV}) reads as
$$
\Big(A(\xi)\;\nxi^2 V(\xi)\Big)_{ij}= (p-1) \left( \partial_{\eta_i}H_0 H_{\xi_j} + HH_{\xi_k\xi_j}\partial^2_{\eta_k\eta_i}\right)H_0 = (p-1)\delta_{ij},
$$
where the last identity follows from (\ref{miracolo}).
%
%
\end{proof}

We notice that, in view of \eqref{HDHo}, \eqref{eq_inversa_D2V} can be alternatively written as
\begin{equation} \label{eq_inversa_D2V_II}
(\nxi^2 V(\xi))^{-1}_{ij}=\frac{H^{-p}(\xi)}{p-1}   \left( \xi_i \xi_j  + (p-1) H^2(\xi) \partial^2_{\eta_i \eta_j} H_0(\nxi H(\xi)) \right)  \,.
\end{equation}

\subsection{Finsler Metric}
For a sufficiently regular set $\Om\subset \rn$ we denote by $P_H(\Om)$ its {\emph{anisotropic perimeter}}, or anisotropic surface energy, that is
\begin{equation}\label{anis_surf_energy}
P_H(\Om)=\int_{\pa \Om}H(\nu) d\haus(x).
\end{equation}
It is clear that if $H$ is the Euclidean norm then $P_H(\Om)$ is the usual perimeter of $\Omega$.

Following \cite[formulae (3.3), (3.9)]{BP}, the \emph{anisotropic mean curvature} of $\partial \Omega$, which we shall denote by $\Huno_{\partial \Omega}$, is defined by
\begin{equation} \label{MHOm}  
\Huno_{\partial \Omega}= - \diver(\nxi H(\nu)) = - H_{\xi_i\xi_j}\nu_i^j \,,
\end{equation}
where $\nu$ is the inward unit normal to $\Omega$.

We notice that if $H(\xi)=|\xi|$ then $\Huno$ is the usual mean curvature normalized so that for the Euclidean unit ball $B$ it holds $\Huno_{\partial B}=(N-1)$.

As it is well known, in the Euclidean setting the only compact connected constant mean curvature hypersurfaces without boundary are Euclidean balls (Alexandrov's Theorem). In the Finsler metric an analogous result holds (see \cite{BCS} and \cite{HLMG}).
\begin{theorem}[Anisotropic Alexandrov's Theorem]\label{AleksandrovThm}
	Let $H$ be a norm of $\rn$ in the class $\hnorm$, $p>1$, and let $\pa \Omega$ be a compact connected hypersurface without boundary embedded in Euclidean space of class $C^2$.
	If $\Huno_{\partial \Omega}$ is constant on $ \pa \Omega$ then $\Omega$ is Wulff shape of $H$.
\end{theorem}

\subsection{Finsler $p$-Laplacian}
The Finsler or anisotropic $p$-Laplacian (defined by (\ref{pLapl})) of a function $u\in C^2(D)$ can be written as
$$
\Delta^H_p u=H^{p-2}(Du)\Big( (p-1)H_{\xi_i}(Du)H_{\xi_j}(Du) + H(Du) H_{\xi_i\xi_j}(Du) \Big) u_{ij},
$$
for $Du \neq 0$.
If $H$ is a norm in the class $\hnorm$ then the Finsler $p$-Laplacian is an elliptic operator and it satisfies
\begin{equation}\label{ellitticita}
\begin{aligned}
& H^{p-2}(\xi)\Big( (p-1)H_{\xi_i}(\xi)H_{\xi_j}(\xi) + H(\xi)H_{\xi_i\xi_j}(\xi) \Big)\eta_i\eta_j \ge c |\xi|^{p-2} |\eta|^2,\\ 
&  \sum_{i,j}H^{p-2}(\xi)| (p-1)H_{\xi_i}(\xi)H_{\xi_j}(\xi) + H(\xi)H_{\xi_i\xi_j}(\xi) | \le c^{-1} |\xi|^{p-2} \,,	
\end{aligned}
\end{equation}
for some $c>0$. Moreover a comparison principle (and hence also a maximum principle) for $\Delta_p^H$ can be derived.

\begin{lemma}\label{lemma_comp_princ}
Let $E \subset \rn$ be a bounded domain and $H$ be a norm in the class $\hnorm$, $p>1$. Let $u,v \in C^1(\overline{E})$ be such that 
\begin{equation} \label{u_e_v}
\begin{cases} 
-\Delta_p^H u \leq - \Delta_p^H v & \text{in } E \,,\\
u \leq v & \text{on } \partial E\,,
\end{cases}
\end{equation}
in weak sense. Then $u \leq v$ in $\overline{E}$.
\end{lemma}

\begin{proof}
Since \eqref{u_e_v} holds in weak sense, we can use $(u-v)_+$ as test function, and hence
\begin{equation} \label{uv_eq1}
\int_{\{u>v\}} \langle H^{p-1}(Du)\, \nxi H(Du) - H^{p-1}(Dv) \,\nxi H(Dv) ; D(u-v) \rangle\,dx \leq 0 \,.
\end{equation}
We notice that \eqref{defH0} and \eqref{HDHo=1} imply
$$
\langle \nxi H(\xi);\eta\rangle \leq H(\eta)
$$
so that from \eqref{Hhomog1} we obtain 
\begin{multline*}
\langle H^{p-1}(Du)\, \nxi H(Du) - H^{p-1}(Dv)\, \nxi H(Dv) ; Du-Dv \rangle \\ = H^{p}(Du) + H^{p}(Dv) - \langle H^{p-1}(Du) \nxi H(Du) ; Dv \rangle - \langle H^{p-1}(Dv) \nxi H(Dv) ; Du \rangle \\
\geq \Big[ H^{p-1}(Du) - H^{p-1}(Dv) \Big]\, [H(Du) - H(Dv)] \geq 0 \,,
\end{multline*}
where in the last inequality we have used that the function $t^{p-1}$ is strictly increasing and hence the strict sign holds whenever $H(Du) \neq H(Dv)$. From \eqref{uv_eq1} we obtain that $H(Du)=H(Dv)$ in $\{u>v\}$, and the conclusion follows by using the strict convexity of $H^p$ and exploiting again \eqref{uv_eq1}. 
\end{proof}

We recall that on the boundary of a super level set $L_t=\{u > t\}$ of a regular function $u$ we have
\begin{equation}\label{HDeltap=M}
\Delta^H_p u = (p-1)H^{p-2}(Du)H_{\xi_k}(Du)H_{\xi_i}(Du)u_{ki}-H^{p-1}(Du)\Huno_{\partial L_t} \quad \text{ on } \partial L_t \,,
\end{equation}
provided that $Du \neq 0$  on $\partial L_t$.
\subsection{Elementary symmetric function of a matrix}\label{section_preliminaries}
Given a matrix $A=(a_{ij})\in\RR^{n\times n}$, for any $k=1,\dots,n$ we denote by $S_k(A)$ the sum of all the principal  minors of $A$ of order $k$.
In particular, $S_1(A)=\tr(A)$, the trace  of $A$, and $S_n(A)= \det (A)$, the determinant of $A$.


We will consider the case $k=2$.
By setting
$$
S^2_{ij}(A)=-a_{ji} + \delta_{ij} \tr A \,,
$$
we can write
\begin{equation}\label{defS2}
S_2(A)=\frac12\sum_{i,j}S^2_{ij}(A)a_{ij}= \frac 12 ((\tr A)^2-\tr(A^2))\,.
\end{equation}

Of particular interest in our approach is $S_2(W)$ where 
$$
W=\nxi^2 V(Dv)D^2v \,,
$$ 
with $V(\xi)=\frac 1pH^p(\xi)$, $H$ a norm in $\hnorm$, $p>1$, and $v$ is a function which will be specified later
($v$ will be either the solution to \eqref{int-pbu} or a suitable power of the solution to \eqref{ext-pbu}).
In this case, since $\tr W=\Delta^H_p v$, it holds
\begin{equation}\label{S2Wij=}
\sdueij{W}= -V_{\xi_j\xi_k}(Dv) v_{ki} +\delta_{ij} \Delta^H_pv \,.
\end{equation}
Moreover in this setting $\sdueij{W}$ is divergence free, in the following (weak) sense (see in \cite[formula (4.14)]{CS} and Lemma \ref{lemma_ident_int_1} in Subsection \ref{subsect_interior} below)
\begin{equation}\label{divS2=0}
\frac{\pa}{\pa x_j}\sdueij{W}=0.
\end{equation}

The elementary symmetric functions of a symmetric matrix $A$ satisfy the so called Newton's inequalities. In particular, we will use that
$$
S_2(A)\le (S_1(A))^2 \,,
$$
which in fact is a straightforward consequence of \eqref{S2Wij=} and Cauchy-Schwarz inequality.
More precisely, we will need a generalization of this inequality to not necessarily symmetric matrices, which is given in the following lemma. We remark that this inequality, together with the characterization of the equality case, is one of the crucial ingredients in the proofs of our main results, since it is the only inequality that we use in our argument to obtain the symmetry result.
\begin{lemma}[\cite{CS}, Lemma 3.2]
\label{lemma-newton}
	Let $B$ and $C$ be symmetric matrices in $\RR^{N\times N}$, and let $B$ be positive semidefinite. Set $A=B\,C$. Then the following inequality holds:
	\begin{equation}\label{newtonIneq}
	S_2{(A)}\le\frac{N-1}{2N}\tr(A)^2\, .
	\end{equation}
	Moreover, if ${\tr} (A)\neq 0$ and equality holds in (\ref{newtonIneq}), then
	\begin{equation*}
	A=\frac{{\tr}(A)}{N}\, I \,,
	\end{equation*}
	and $B$ is, in fact, positive definite.
\end{lemma}
Notice that we are going to apply \eqref{newtonIneq} to the matrix $W=\nxi^2V(Dv)D^2v$ where $V$ is as in (\ref{V_def}) and $\nxi^2V$ is in fact positive definite since $H\in\hnorm$.


\section{Preliminary results} \label{sect_prelim_II}
We recall that $H$ is a norm, so that $H$ satisfies (i)--(iii) in Subsection \ref{section_norms}, and $H \in \hnorm$, $p>1$, where $\hnorm$ is given by \eqref{Ip_def}.

\subsection{Interior problem}
We start with the following lemma, which describes the expected regularity of the solution to the interior problem \eqref{int-pbu}.

\begin{lemma} \label{lemma_regularity_int}
Let $\Omega\subset\rn$ be a bounded domain with boundary of class $C^{2,\alpha}$ and let $H$ be a norm in $\hnorm$, $p>1$. There exists a unique solution $u$ to Problem (\ref{int-pbu}) with $u\in C^{1,\alpha}(\overline{\Omega})$ and
\begin{equation} \label{Wu_reg}
H^{p-1}(Du)\nabla_{\xi}H(Du) \in W^{1,2}(\Omega) \,.
\end{equation}
Moreover, if $1 < p \leq 2$ then $u \in W^{2,2}(\Omega)$.
\end{lemma}
\begin{proof}
The statement is a collection of well-known results in regularity theory and we give only a sketch of the proof.  Since $H \in \mathcal{I}_p$ then \eqref{ellitticita} holds and the global $C^{1,\alpha}$ regularity of the solution follows from  \cite[Theorem 1]{Lieberman} (see also \cite{DiB} and \cite{To}).
\par
In order to prove \eqref{Wu_reg}, we first notice that from \cite[Theorem 4.1]{AKM} we have that 
\begin{equation} \label{Wreg_loc_II} 
H^{p-1}(Du)\nabla_{\xi}H(Du) \in W^{1,2}_{\rm loc}(\Omega) \,.
\end{equation}
We notice that \eqref{Wreg_loc_II} is obtained in \cite{AKM} for solutions to homogeneous equations, but the argument used in the proof can be easily adapted to the case of constant right hand side and we omit the proof.

Since $\partial \Omega$ is of class $C^{2,\alpha}$, then a standard barrier argument shows that there exists $0<c<1$ such that $c \leq |Du| \leq 1/c$ on $\partial \Omega$. Being $u \in C^{1,\alpha}(\Omega)$, this entails that
$c/2 < |Du|< 2/c$ in a neighborhood $\widetilde \Omega$ of $\partial \Omega$, which implies that $u$ solves a uniformly elliptic equation in $\widetilde \Omega$. From classical regularity theory we obtain that $u \in C^{2,\alpha}$ in $\widetilde \Omega$ which, together with \eqref{Wreg_loc_II}, implies \eqref{Wu_reg}.
\par
Finally, if $1 < p \leq 2$ then $u  \in W_{\rm loc}^{2,2}(\Omega)$ (which can be obtained by generalizing \cite[Proposition 2.7]{AF} to an equation with constant right hand side) and the global $W^{2,2}$-regularity follows again  by exploiting the $C^{2,\alpha}$ regularity of $u$ in a suitable neighborhood of $\partial \Omega$.
\end{proof}	

Now we prove the following Poho\v{z}aev identity which will be used to achieve the Wulff shape characterization in Theorem \ref{thm_serrin_int}.

\begin{lemma}[Anisotropic Poho\v{z}aev identity]\label{lemma-pohoz}	
Let $u$ be the solution to (\ref{int-pbu}). We have that
\begin{equation} \label{pohoz}
\Big(N(p-1) + p\Big) \int_{\Omega}u \,dx = -(p-1) \int_{\partial \Omega} H^p(Du)\, \langle x;\nu\rangle\, d\sigma \,.
\end{equation}
\end{lemma}

\begin{proof}
As usual, the proof of the Poho\v{z}aev identity is obtained by integrating a differential identity involving the second derivatives of $u$. However, due to the lack of enough regularity of $u$, we shall argue by approximation. 

\emph{Step 1: set up of the approximation argument}. For $ t \geq 0$ and $\varepsilon \in (0,1)$, we set 
\begin{equation*} 
\label{psi_ep}
\psi(t)=\frac{t^p}{p} \qquad \mbox{ and } \qquad \psi_\ep(t) = \psi\left(\sqrt{\ep^2 + t^2}\right) - \psi(\ep) \,.
\end{equation*}
We define $\Psi(t)=\psi'(t)\,t$ and $\Psi_\ep(t)=\psi'_\ep(t)\,t$. From a standard argument (see for instance \cite[Lemma 4.2]{CFV2}) we have that

\begin{equation} \label{psi_ep_conv}
\psi_\ep \to \psi \qquad \mbox{and} \qquad \Psi_\ep \to \Psi \quad  \text{ uniformly on compact sets of } [0,+\infty) \,. 
\end{equation}
We notice that $V(\xi) = \psi(H(\xi))$, $\xi \in \rn$, and we define $V^\ep: \rn \to \RR$ as 
\begin{equation*} \label{V_ep_def}
V^\ep:= \psi_\ep \circ H \,.
\end{equation*}
Let $u^\ep$ be the solution to the problem
\begin{equation*} \label{u_ep_pbu}
\begin{cases}
\diver\left(  \psi_\ep'(H(Du^\ep)) \nabla_\xi H(Du^\ep) \right)= -1 & \textmd{in } \Omega \,, \\
u^\ep=0 & \textmd{on } \partial \Omega \,,
\end{cases}
\end{equation*}
or equivalently
\begin{equation*} \label{u_ep_pbu_I}
\begin{cases}
\diver \left(  \nabla_\xi V^\ep(Du^\ep(x)) \right)= -1 & \textmd{in } \Omega \,, \\
u^\ep=0 & \textmd{on } \partial \Omega \,.
\end{cases}
\end{equation*}
We notice that for $\ep>0$ the above problems are uniformly elliptic; in particular they satisfy
$$
\lambda\,\left(\varepsilon+|z|^2\right)^\frac{p-2}{2}\,|\xi|^2\le \langle \nabla^2 V^\varepsilon(z)\,\xi,\xi\rangle\le \Lambda\,\left(\varepsilon+|z|^2\right)^\frac{p-2}{2}\,|\xi|^2,\qquad \mbox{ for } z,\xi\in\mathbb{R}^N,
$$
for some $0<\lambda<\Lambda$ which depend only on $p$ and on the ellipticity constants of $H$. In particular, they are independent of $\varepsilon$.

Standard regularity results give that $u^\ep \in C^{1,\alpha}(\overline\Omega) \cap W^{2,2}(\Omega)$, $u^\ep$ is a strong solution and 
\begin{equation} \label{u_ep_to_u_C1}
u^\ep \to u \quad \text{  in } C^1(\overline \Omega),
\end{equation} 
as $\ep$ goes to $0$ (see for instance \cite[Proposition 4.3]{CFV2}).

\emph{Step 2: proof of \eqref{pohoz}}. Now we are ready to prove the Poho\v{z}aev identity. We notice that $\int_{\Omega}\diver(x\,u^\ep)\,dx=0$ since $u^\ep=0$ on $\bd\Omega$ and hence
$$
N\int_{\Omega}u^\ep\, dx  =\int_{\Omega}(-\diver(x\,u^\ep)+N\,u^\ep)\,dx \,,	
$$
which implies
\begin{equation}\label{phoz3}
	N\int_{\Omega}u^\ep dx =-\int_{\Omega}\langle x;Du^\ep \rangle dx \,.
\end{equation}
Since $-1=\diver(\nabla_\xi V^\ep(Du^\ep))$ in $\Omega$, we have (recall that $u^\ep \in W^{2,2}(\Omega)$)
\begin{equation*} 
\begin{aligned}
-\int_{\Omega}& \langle x;Du^\ep\rangle  \;dx  =\int_{\Omega}\langle x;Du^\ep\rangle\diver(\nabla_\xi V^\ep(Du^\ep)) \;dx \\
& =\int_{\Omega}\left[ \diver(\langle x;Du^\ep \rangle \nabla_\xi V^\ep(Du^\ep))-\langle D(\langle x;Du^\ep \rangle); \nabla_\xi V^\ep(Du^\ep) \rangle \right] dx  \\
&=-\int_{\partial \Omega}\langle x;Du^\ep \rangle \langle \nabla_\xi V^\ep(Du^\ep); \nu \rangle - \int_\Omega \langle Du^\ep ; \nabla_\xi V^\ep(Du^\ep) \rangle dx - \int_\Omega \langle x D^2u^\ep; \nabla_\xi V^\ep(Du^\ep) \rangle dx \,;
\end{aligned}
\end{equation*}
from the definition of $V^\ep$, \eqref{Hhomog1} and being $\nu=D u^\ep/|Du^\ep|$, we obtain 
\begin{multline}   \label{phoz4}
-\int_{\Omega} \langle x;Du^\ep\rangle  \;dx  = -\int_{\partial \Omega}\psi_\ep'(H(Du^\ep)) H(Du^\ep) \langle x;\nu \rangle - \int_\Omega \psi_\ep'(H(Du^\ep))  H(Du^\ep)  dx \\ - \int_\Omega \langle x D^2u^\ep; \nabla_\xi V^\ep(Du^\ep) \rangle dx \,.
\end{multline}
Notice that the last term on the right hand side in \eqref{phoz4} can be written as 
\begin{equation*}
\begin{aligned}
- \int_\Omega \langle x D^2u^\ep; \nabla_\xi V^\ep(Du^\ep) \rangle dx & = - \int_\Omega \diver (x V^\ep(Du^\ep)) dx + N \int_\Omega V^\ep(Du^\ep) dx \\
& =  \int_{\partial \Omega} V^\ep(Du^\ep) \langle x; \nu \rangle dx + N \int_\Omega V^\ep(Du^\ep) dx \,,
\end{aligned}
\end{equation*}
and from \eqref{phoz3} and \eqref{phoz4} we have 
\begin{multline*}
N\int_{\Omega}u^\ep dx = -\int_{\partial \Omega}\psi_\ep'(H(Du^\ep)) H(Du^\ep) \langle x;\nu \rangle - \int_\Omega \psi_\ep'(H(Du^\ep))  H(Du^\ep)  dx \\  + \int_{\partial \Omega} V^\ep(Du^\ep) \langle x; \nu \rangle dx + N \int_\Omega V^\ep(Du^\ep) dx  \,.
\end{multline*} 
Now we use \eqref{u_ep_to_u_C1} and \eqref{psi_ep_conv} to pass to the limit as $\ep \to 0$ and we find 
\begin{equation} \label{poho_quasi}
N\int_{\Omega}u dx = \left(-1+ \frac1p \right) \int_{\partial \Omega} H^p(Du) \langle x;\nu  \rangle    + \left( \frac Np - 1\right) \int_\Omega H^p(Du) dx \,.
\end{equation} 
Finally, by using $u$ as a test function in \eqref{int-pbu_distr} we have that 
$$
\int_\Omega u = \int_\Omega H(Du)^p \,,
$$		
and from \eqref{poho_quasi} we obtain \eqref{pohoz}.
\end{proof}

\subsection{Exterior problem}
In this subsection we give some preliminary result related to the solution to \eqref{pcap}.

\begin{theorem}\label{teostime}
Let $\Omega$ be a bounded convex domain of $\rn$ whose boundary is of class $C^{2,\alpha}$ and assume $O\in\Omega$.
Let $H$ be a norm of $\rn$ in the class $\hnorm$, $p>1$.

There exists a unique solution $u$ to Problem \eqref{pcap}, $u\in C^{2,\alpha}(\rn\setminus {\Omega})$, and $0<u \leq 1$. Moreover $u$ satisfies \eqref{ext-pbu} and the following estimates hold:
\begin{itemize}
\item[(i)] there exist $A_1,A_2$ positive constants such that
$$
A_1 H_0^{\frac{p-N}{p-1}}(x)\le u(x)\le A_2 H_0^{\frac{p-N}{p-1}}(x),
$$
for $x\in\rn\setminus \Omega$;
\item[(ii)] there exist $B_1,B_2$ positive constants such that 
$$
B_1 H_0^{\frac{p-N}{p-1}-1}(x)\le H(Du(x))\le B_2 H_0^{\frac{p-N}{p-1}-1}(x),
$$
for $x$ sufficiently far away from $\Omega$;
\item[(iii)] there exists a positive constant $B_3$ such that
$$
|D^2u(x)|\le {B_3}{H_0^{\frac{p-N}{p-1}-2}}(x),
$$
for $x$ sufficiently far away from $\Omega$.
\end{itemize}
The constants $A_1,A_2,B_1,B_2,B_3$ depend only on $\Omega$, $p$ and $N$.
\end{theorem}
\begin{proof}
Let $R>0$ be such that $\bhor{R}\supset\overline{\Omega}$ and let  $u_R$ be the minimizer of
\begin{equation} \label{cap_rel}
\pCap(\Omega, B_{H_0}(R))=\inf \left\{  \frac 1p\int_{\bhor{R}} H^p(D\varphi)\, dx\, :\, \ \varphi\in C^{\infty}_0(\bhor{R}), \varphi(x)\ge 1 \text{ for }x\in\Omega \right\} \,.
\end{equation}
Since $H\in\hnorm$, a standard argument yields that there exists a unique minimizer $u_R$ and it solves the Euler-Lagrange equation
\begin{equation}\label{capBR}
\begin{cases}
\Delta^H_p u_R=0 &\qquad\text{ in }\bhor{R}\setminus\overline{\Omega} \,,\\
u_R=1 &\qquad\text{ on }\partial\Omega \,, \\
u_R=0 &\qquad\text{ on } \partial\bhor{R} \,.
\end{cases}
\end{equation}
Thanks to the comparison principle in Lemma \ref{lemma_comp_princ}, if $r>s$ then $u_{r}(x)\ge u_{s}(x)$ for every $x\in\bhor{s}\setminus\overline{\Omega}$ and hence the function 
$$
u(x)=\lim_{R\to\infty}u_R(x)
$$
is well defined, for $x\in\rn\setminus\overline{\Omega}$
and the sequence $u_R$ is in fact uniformly convergent. We are going to prove (i) and the lower bound in (ii) for $u_R$ and show that the involved constants do not depend on $R$, so that we obtain the desired estimates for $u$ by passing to the limit as $R\to\infty$. The upper bounds in (ii) and (iii) will be obtained by arguing directly on $u$.

Let $0<R_0<R_1$ be such that
$$
R_0=\sup\{r>0\ :\ \bhor{r}\subset\Omega\}; \qquad R_1=\inf\{r>0\ :\ \Omega\subset\bhor{r}\},
$$
and let $u_{R_0,R}$, $u_{R_1,R}$ be solutions to (\ref{capBR}) for $\Omega=\bhor{R_0}$ and $\Omega=\bhor{R_1}$, respectively.
By comparison principle it holds
$$
u_R\ge u_{R,R_0}= \frac{H_0^\pesp(x)-R^\pesp}{R_0^\pesp-R^\pesp},
$$
for every $x\in\bhor{R}\setminus\Omega$ and
$$
u_R\le u_{R,R_1}= \frac{H_0^\pesp(x)-R^\pesp}{R_1^\pesp-R^\pesp},
$$
for every $x\in\bhor{R}\setminus{\bhor{R_1}}$.
Notice that in fact the latter inequality holds true in $\bhor{R}\setminus{\Omega}$ by direct comparison between $u_R$ and $u_{R,R_1}$.
Hence (i) holds for $u$ for every $x\in\rn\setminus\overline{\Omega}$ by passing to the limit as $R\to\infty$.

Since $H\in\hnorm$ and $\Delta^H_p$ satisfies (\ref{ellitticita}), classical regularity results for degenerate elliptic equations in divergence form with potential growth (see \cite{DiB}, \cite{Lieberman} and \cite{To})  guarantee that $u_R\in C^{1,\alpha}(\bhor{R}\setminus{\Omega})$, where $\alpha$ does not depend on $R$.

Moreover, by Lemma \ref{lemmaLewis} it holds $H(Du_R)\neq 0$ and hence Theorem 6.19 in \cite{GT} entails that $u_R\in C^{2,\alpha}(\overline{\bhor{R}\setminus{\Omega}})$.
More precisely Lemma \ref{lemmaLewis} gives the lower bound on $u$ in (ii).

Now we prove the upper bound on $u$ in (ii) and (iii). Let $\rho>4R_1$ be fixed. For $y \in E:= \overline{\bhor{4}} \setminus \bhor{1/4}$ we define 
\begin{equation} \label{Ugrande}
U(y)=\rho^{-\frac{p-N}{p-1}} u(\rho y) \,,
\end{equation}
and notice that
\begin{equation*} \label{Ueq}
\Delta^H_p U =0 \qquad \text{in }E \,.
\end{equation*}
Moreover, from (i) we have that $|U(y)| \leq A$ for $y\in E$ and for some constant $A$ which depends only on $n,p$ and $\Omega$. From \cite[Theorem 1]{DiB}, there exists a constant $K$ depending only on $\Omega$, $n$ and $p$ such that $|DU(y)| \leq K$ for $y\in \overline{\bhor{2}} \setminus \bhor{1/2}$. Being $DU(y) = \rho^{-\frac{p-N}{p-1}+1} Du(\rho y)$ and from \eqref{norms equiv} we obtain the upper bound in (ii).

As noticed in \eqref{Ugrande}, $U$ satisfies an elliptic equation of the form $a_{ij} U_{ij} = 0$ in $E$, where the coefficients $a_{ij}$ are given by
$$
a_{ij}(y) = H^{p-2}(DU(y))\Big( (p-1)H_{\xi_i}(DU(y))H_{\xi_j}(DU(y))+H(DU(y))H_{\xi_i\xi_j}(DU(y))  \Big) \,.
$$
From (ii) and \eqref{norms equiv} we have that there exists $\gamma$ depending only on $N$, $p$ and $\Omega$ such that 
$$
\gamma^{-1}  |\xi|^2 \leq a_{ij}(y)\xi_i\xi_j\le \gamma |\xi|^2
$$ 
for every $y\in E$ and $\xi\in\rn$.
Notice that 
interior Schauder's estimates (see Theorem 6.2 in \cite{GT}) apply to $U(y)$.
This entails $|D^2U(y)|\le B$ for some positive constant $B$, that is
$$
|D^2u(\rho y)|\le B \rho^{\frac{p-N}{p-1}-2},
$$
for $y\in \bhor{2}\setminus\overline{\bhor{1/2}}$ and (iii) follows.
\end{proof}


\section{Step 1 - Integral identities for $S_2$} \label{section4}
In this section we achieve Step 1, that is we derive two crucial integral identities involving $S_2$ and the solutions to problems \eqref{int-pbu} and \eqref{ext-pbu}. More precisely, we use the pointwise identity in Lemma \ref{lemmaS2W=} below to obtain  an integral identity for $u$ in Lemma \ref{lemma_ident_int} (for the interior problem) and for $u^{\frac{p}{p-N}}$ in Lemma \ref{lemma_ident_ext} (for the exterior case).

In the following lemma we assume that all the involved functions are smooth enough. In particular, we stress that $V$ is not necessarily related to a norm and hence is not in principle homogeneous.

\begin{lemma}\label{lemmaS2W=}
	Let $v$ be a positive function of class $C^3$ and let $V:\rn\to \RR^+$ be of class $C^3(\rn)$ and such that $V (Dv) \diver \left(\nxi V(Dv) \right)$ can be continuously extended to zero at $Dv=0$. For any $\gamma \in \RR$ we have that
\begin{equation}\label{2vgammas2}
	2v^{\gamma}S^2(W)= \diver(v^\gamma\sdueij{W}V_{\xi_i}(Dv))-\gamma v^{\gamma-1}\sdueij{W}V_{\xi_i}(Dv)v_j \,,
\end{equation}	
with $W=\nabla^2_\xi V(Dv)D^2v$.

Moreover, if $H$ is a norm and $V=H^p/p$, $p>1$, then
	\begin{eqnarray}\label{S2W=equazione-puntuale-norma}
	2v^{\gamma}S^2(W)&=& \diver(v^\gamma\sdueij{W}V_{\xi_i}(Dv)+ \gamma(p-1)v^{\gamma-1}V(Dv)\nabla_\xi V(Dv))\\
	&& - \gamma(\gamma-1)p({p-1}) v^{\gamma-2}V^{2}(Dv)-\gamma ({2p-1})v^{\gamma-1}V(Dv)\Delta^H_p v \,. \nonumber
	\end{eqnarray}
\end{lemma}

\begin{proof}	
For simplicity of exposition, we omit the dependency on $Dv$ in the argument of $V$ and $H$, so that $H$ and $V$ will be always evaluated at $Dv$.

From (\ref{defS2}) and (\ref{divS2=0}) it holds that $2S^2(W)=\diver(\sdueij{W}V_{\xi_i})$, and hence we find \eqref{2vgammas2}.
Moreover, we have
\begin{multline}\label{divervgammaS2}
	\diver\Big(v^\gamma\sdueij{W}V_{\xi_i}+ \gamma(p-1)v^{\gamma-1}V\nabla_\xi V\Big ) \\ = 2v^{\gamma}S^2(W) + \gamma v^{\gamma-1} \sdueij{W}V_{\xi_i}v_j+\gamma(p-1)\diver(v^{\gamma-1}V\nabla_\xi V).
	\end{multline}
From the definition of  $\sdueij{W}$ (\ref{S2Wij=}) we find
	$$
	\sdueij{W}V_{\xi_i}v_j=-V_{\xi_j\xi_l}v_{l i}V_{\xi_i}v_j+ V_{\xi_i}v_i \tr(W)
	$$
	and, since
	$$
	\diver(v^{\gamma-1}V\nabla_\xi V)=(\gamma-1)v^{\gamma-2}VV_{\xi_i}v_i+v^{\gamma-1}V_{\xi_i}V_{\xi_j}v_{ij}+v^{\gamma-1}V\tr (W) \,,
	$$
	from (\ref{divervgammaS2}) we obtain
	\begin{eqnarray} \label{pippo}
	2v^{\gamma}S^2(W) &=& \diver\Big(v^\gamma\sdueij{W}V_{\xi_i}+ \gamma(p-1)v^{\gamma-1}V\nabla_\xi V\Big)\nonumber \\
	&&  -\gamma(\gamma-1)(p-1)v^{\gamma-2}VV_{\xi_i}v_i - \gamma v^{\gamma-1}((p-1)V+V_{\xi_i}v_i) \tr(W)  \\\nonumber
	&& - \gamma\; v^{\gamma-1}\Big( (p-1)V_{\xi_i}V_{\xi_j}v_{ij}+V_{\xi_j\xi_l}v_{l i}V_{\xi_i}v_j\Big) \,.
	\end{eqnarray}
In the case $V=H^p/p$ and $H$ is a norm, the last term on the right hand side reads
\begin{multline} \label{freddo}
(p-1)V_{\xi_i}V_{\xi_j}v_{ij}-V_{\xi_j\xi_l}v_{li}V_{\xi_i}v_j  \\ = (p-1) H^{2(p-1)} H_{\xi_j} H_{\xi_i} v_{ij} - H^{p-1}\left( (p-1) H^{p-2} H_{\xi_j} H_{\xi_l} + H^{p-1} H_{\xi_j \xi_l}\right) H_{\xi_i} v_j v_{li} = 0\,,
\end{multline}
where the last equality follows from \eqref{Hhomog1} and \eqref{Hhomog2}, and from \eqref{pippo} we obtain \eqref{S2W=equazione-puntuale-norma}.
\end{proof}

\subsection{Interior problem} \label{subsect_interior}
We use Lemma \ref{lemmaS2W=} to obtain an integral identity for the solutions of \eqref{int-pbu}. 
The integration by parts formula in the following lemma was already obtained in  \cite[Lemma 4.3]{CS} for the case $p=2$. The proof in \cite{CS} makes use of the $W^{2,2}$ regularity of the solution $u$, which is not available in the general case $p>1$. For this reason, we argue in a different way.

\begin{lemma} \label{lemma_ident_int_1}
	Let $\Omega\subset\rn$ be a bounded domain with boundary of class $C^{2,\alpha}$. Let $H$ be a norm in $\hnorm$, $p>1$, and let $u$ be the solution to (\ref{int-pbu}). Then the identity
	\begin{equation}\label{S2W=equazione-interno}
	\int_\Omega 2 \phi S^2(W)\, dx = -\int_\Omega \sdueij{W}V_{\xi_i}(Du) \phi_j \, dx \,,
	\end{equation}	
holds for every $\phi \in C^1(\overline \Omega)$ such that $\phi=0$ on $\partial \Omega$, with $W=\nabla^2 V(Du)D^2u$, where $V$ is given by \eqref{V_def}.
\end{lemma}

\begin{proof}
Let $\ep>0$ be sufficiently small and define $\Omega_\ep = \{x \in \Omega:\ {\rm dist}(x,\partial \Omega)>\ep\}$. Let 
$$
a^i (x) = V_{\xi_i}(Du(x))  \quad \text{ for every } i=1,\ldots,N\,,\ x \in \Omega\,.
$$
We mention that $a^i \in W^{1,2}(\Omega)$, $i=1,\ldots,N$, as follows from Lemma \ref{lemma_regularity_int}. With this notation, the elements $w_{ij}$ of the matrix $W$ are given by $w_{ij}=\partial_j a^i$. Let 
$\rho_\ep$ be a family of mollifiers and define $a_\ep^i = a^i \ast \rho_\ep $.
 Let $W^{\ep}=(w^\ep_{ij})_{i,j=1,\ldots,N}$ where $w_{ij}^\ep = \partial_j a^i_\ep$, and notice that 
\begin{equation} \label{trace_W_conv}
\tr  W^{\ep}= \tr W = -1 
\end{equation}
for every $x \in \Omega_\ep$.

Let $i,j=1,\ldots,N$ be fixed. We have
\begin{equation*}
\begin{split}
w_{ji}^\ep  w_{ij}^\ep & =  \partial_j (a^i_\ep \partial_i a^j_\ep) - a^i_\ep \partial_j \partial_i a^j_\ep \\
& =   \partial_j (a^i_\ep \partial_i a^j_\ep) - a^i_\ep \partial_i \partial_j a^j_\ep  \\
& =\partial_j (a^i_\ep \partial_i a^j_\ep) -  a^i_\ep \partial_i w_{jj}^\ep \,,
\end{split}
\end{equation*}
for every $x \in \Omega_\ep$, and by summing over $j=1,\ldots,N$ and using \eqref{trace_W_conv} (so that  $\partial_i \sum_j w_{jj}^\ep=0$) we obtain 
\begin{equation*}
\begin{split}
\sum_j w_{ji}^\ep  w_{ij}^\ep &  = \sum_j \partial_j (a^i_\ep \partial_i a^j_\ep) \\
& = w_{ii}^\ep \tr W^\ep - \sum_j \partial_j( S^2_{ij}(W^\ep) a_\ep^i )  \,, \quad x \in \Omega_\ep \,.
\end{split}
\end{equation*}
By summing over $i=1,\ldots,N$, from \eqref{defS2} we have
\begin{equation}\label{S2ep}
2 S^2(W^\ep)= \sum_{i,j} \partial_j( S^2_{ij}(W^\ep) a^i_\ep )  \,, \quad x \in \Omega_\ep \,.
\end{equation}
Let $\ep_0>0$ be such that $u \in C^{2,\alpha}$ in $\bar \Omega \setminus \Omega_{\ep_0}$ (this is always possible since $H(Du)>0$ on $\partial \Omega$). We notice that, by a standard barrier argument, one can obtain bounds on $Du$ on $\partial \Omega$ and, since $u=0$ on $\partial \Omega$, the equation gives a bound on $D^2u$ on $\partial \Omega$. Thanks to the $C^{2,\alpha}$ regularity of $u$, we  obtain a bound on the $C^2$ norm of $u$ in $\bar \Omega \setminus \Omega_{\ep_0}$ which does not depends on $\epsilon_0$.

Let $\ep <\ep_0$. An integration by parts and \eqref{S2ep} give
$$
\Big{|} 2 \int_{\Omega_{\ep_0}} \phi S^2(W^\ep) +\int_{\Omega_{\ep_0}} \sdueij{W^\ep}a^i_\ep \phi_j \, dx \Big{|} = \Big{|}  \int_{\partial \Omega_{\ep_0}} \phi \sdueij{W^\ep}a^i_\ep \nu_j  \Big{|} \leq c \ep_0  \,,
$$
where $c$ depends on $\|\phi\|_{C^1(\overline \Omega)}$ and bounds on $H(Du)$ on $\partial \Omega$. The assertion follows by letting firstly $\ep$ and then $\ep_0$ to zero.
\end{proof}

\begin{lemma} \label{lemma_ident_int}
	Let $\Omega\subset\rn$ be a bounded domain with boundary of class $C^{2,\alpha}$. Let $H$ be a norm in $\hnorm$, $p>1$, and let $u$ be the solution to (\ref{int-pbu}). Then the identity
	\begin{equation}\label{S2W=equazione-interno_II}
	\int_\Omega 2u S^2(W)\, dx = -\frac{(p-1)}p \int_{\bd \Omega}  H^{2p-1}(Du) H(\nu) \,d\sigma(x) +\frac{(2p-1)}p\int_\Omega  H^p(Du)\, dx
	\end{equation}	
holds with $W=\nabla^2 V(Du)D^2u$, where $V$ is given by \eqref{V_def}.
\end{lemma}

\begin{proof} 
From 
$$
\diver(V\nabla_\xi V)= V_{\xi_i}V_{\xi_j}v_{ij}+V\tr (W) = V_{\xi_i}V_{\xi_j}v_{ij} - V  
$$
by multiplying by $(p-1)$ and using that $\tr(W)=-1$, we have
$$
(p-1)\int_{\partial \Omega} V\langle \nabla_\xi V; \nu \rangle  = (p-1)\int_\Omega \left(V_{\xi_i}V_{\xi_j}v_{ij} - V  \right)\,.
$$
We sum this identity and \eqref{S2W=equazione-interno} with $\phi=u$ and we find
\begin{equation} \label{efinita}
\int_\Omega 2 u S^2(W) = \int_\Omega \left\{ (p-1) \left[V_{\xi_i}V_{\xi_j}v_{ij} - V  \right]
 - \sdueij{W}V_{\xi_i} u_j \right\} - (p-1)\int_{\partial \Omega} \langle \nabla_\xi V; \nu \rangle  \,. 
 \end{equation}

Analogously to what we did in the proof of Lemma \ref{lemmaS2W=}, we obtain that 
$$
(p-1) \left[V_{\xi_i}V_{\xi_j}v_{ij} - V  \right] - \sdueij{W}V_{\xi_i} u_j = \frac{(2p-1)}p  H^p(Du) \quad \text{ a.e. in } \Omega\,,
$$
and from \eqref{efinita}, $\nu=Du/|Du|$ and \eqref{Hhomog1}, we find \eqref{S2W=equazione-interno_II}.
\end{proof}

An straightforward consequence of Lemma \ref{lemma_ident_int} is the following corollary.

\begin{cor}\label{int-condiz-int}
Let $\Omega,\, H\,, u, $ be as in Lemma \ref{lemma_ident_int}. Then
	\begin{equation} \label{ineq_interior}
	N(p-1)\int_{\bd\Omega}H(\nu)H^{2p-1}(Du)\; d\sigma \geq \Big(N(p-1)+p\Big) \int_{\Omega} u \;dx \,,
%
%
	\end{equation}
where the equality sign is attained if and only if there exists a constant $\lambda$ such that $W(x)=\lambda Id$.
\end{cor}
\begin{proof}
We notice that from \eqref{int-pbu} we have that 
$$
\int_\Omega u \; dx = \int_\Omega H^p(Du) \; dx \,.
$$	
From (\ref{S2W=equazione-interno}) and Lemma \ref{lemma-newton}, recalling that $\tr(W)=\Delta^H_pu$ and that $u$ solves Problem (\ref{int-pbu}), we immediately obtain \eqref{ineq_interior}. From Lemma \ref{lemma-newton} we find that the equality is attained if and only if there exists a function $\lambda(x)$ such that $W=\lambda(x) Id$. Moreover, since $u$ satisfies \eqref{int-pbu}, then $\tr(W)=-1$, which implies that $\lambda(x)$ must be constant.
\end{proof}

\subsection{Exterior problem}\label{subsect_exterior}
We apply the machinery described in Subsection \ref{subsect_ideas} to the auxiliary function $v$ defined by
\begin{equation} \label{v_def}
v(x)=u^{\frac p{p-N}}(x) \,,
\end{equation}
$x \in \rn \setminus \Omega$, where $u$ is the solution to \eqref{ext-pbu}. This choice is motivated by the following argument: if $\Omega=B_{H_0}$ then $u(x)=H_0^{\frac{p-N}{p-1}}(x)$ and $v(x)=H_0^{\frac{p}{p-1}}(x)$. This implies that if $\Omega=B_{H_0}$ then $\nabla_\xi^2 V(Dv) D^2 v$ achieves the equality sign in \eqref{newtonIneq}.

Since $u$ solves \eqref{ext-pbu}, straightforward computations show that $v$ satisfies
\begin{equation}\label{ext-pbv}
\begin{cases}
\Delta^H_p v =N\dfrac{p-1}p \dfrac{H^p(Dv)}{v}\qquad&\text{in }\rn\setminus\overline{\Omega},\\
v=1\qquad&\text{on }\bd\Omega \,,\\
v\to +\infty\qquad&\text{if } |x|\to +\infty \,.
\end{cases}
\end{equation}
Moreover, we notice that the Neumann boundary condition $H(Du)=C$ implies
\begin{equation}\label{HDvcost}
H(Dv)=\frac p{N-p}C \qquad\text{on }\bd\Omega \,,
\end{equation}
where we have used $H(Dv)=H(-Dv)$.
We stress that, by geometric reasons, the constant $C$ is forced to be
\begin{equation}\label{valueC}
C= \frac{N-p}{N(p-1)}\frac{P_H(\Omega)}{|\Omega|}.
\end{equation}
A proof of this fact can be found in Appendix \ref{appendix-valoreC}.

Now we use Lemma \ref{lemmaS2W=} to obtain an integral identity for $v=u^{\frac{p}{p-N}}$.

\begin{lemma} \label{lemma_ident_ext}
	Let $\Omega\subset\rn$ be a bounded convex domain with boundary of class $C^{2,\alpha}$.
	Let $H$ be a norm in $\hnorm$, $p>1$, $V$ as in (\ref{V_def}), and let $v$ be given by \eqref{v_def}. Then we have
	\begin{eqnarray}\label{S2W=equazione-esterno}
	\int_{\rn\setminus\overline{\Omega}} 2v^{\gamma} S^2(W)\, dx& =&	 N(N-1)\frac{(p-1)^2}{p^2}\int_{\rn\setminus\overline{\Omega}}H^{2p}(Dv)v^{\gamma-2}\, dx\\ \nonumber
	&& -(N-1)\int_{\bd\Omega} {H^{2p-2}(Dv)} H(\nu) \Big( \frac{\Huno(\bd\Omega)}{N-1} -\frac{p-1}p\frac{H(Dv)}v \Big)\, d\sigma \,,
	\end{eqnarray}	
where $W=\nabla^2 V(Dv)D^2v$ and $\gamma=(1-N)$.
\end{lemma}
\begin{proof}
Starting from Lemma  \ref{lemmaS2W=} we argue by approximation.
%
Let $V_k(\xi):\rn\to \RR^+$ be a sequence of $C^3$ functions which approximate $V$ in the norm $C^{2,\alpha}$.  
Fix $R>0$ such that $\overline{\Omega} \subset \bhor{R}$ and set $\gamma=1-N$. Again, $V$ and $V_k$ we will always evaluated at $Dv$ and we  omit this dependency. We first notice that if $v \in C^3(\bhor{R} \setminus \overline{\Omega}) \cap C^{2} (\overline{\bhor{R} \setminus \Omega})$ then by integrating (\ref{2vgammas2}) we obtain
\begin{multline} \label{formulone}
\int_{\bhor{R} \setminus \Omega} 2v^{\gamma}S^2(W_k) =
	 -\gamma(\gamma-1)(p-1) \int_{\bhor{R} \setminus \Omega}  v^{\gamma-2} V_k (V_k)_{\xi_i}v_i  \\ - \gamma \int_{\bhor{R} \setminus \Omega} v^{\gamma-1}((p-1)V_k+(V_k)_{\xi_i}v_i) \tr(W_k)
	 \\ + \gamma \int_{\bhor{R} \setminus \Omega}  v^{\gamma-1}\Big( (p-1)(V_k)_{\xi_i}(V_k)_{\xi_j}v_{ij}-(V_k)_{\xi_i\xi_l}v_{lj}(V_k)_{\xi_i}v_j \Big) \	\\ - \int_{\partial \Omega} \langle v^\gamma\sdueij{W_k}(V_k)_{\xi_i} + \gamma(p-1)v^{\gamma-1} V_k \nabla_\xi (V_k) ; \nu \rangle
\\ \int_{\partial \bhor{R}} \langle v^\gamma\sdueij{W_k}(V_k)_{\xi_i}+ \gamma(p-1)v^{\gamma-1} V_k \nabla_\xi (V_k) ; \nu \rangle \,,
\end{multline}
where we set $W_k = \nabla^2V_k (Dv) D^2v$. 
We notice that \eqref{formulone} still holds for $v=u^{\frac{p}{p-N}}$ where $u$ is the solution of \eqref{ext-pbu}. Indeed, in this case $v \in C^{2,\alpha}(\bhor{R} \setminus \Omega)$ with $H(Dv)$ bounded away from zero (see Theorem \ref{teostime}). Hence, \eqref{formulone} for $v=u^{\frac{p}{p-N}}$ is obtained by approximating $v$ in $C^{2,\alpha}$ by a sequence of $C^3$ functions.

Now, we notice that Theorem \ref{teostime} implies that the pointwise convergence of the elements $\nabla_{\xi}V_k, \nabla_\xi^2V_k, W_k$ is in fact uniform as $k \to \infty$. Hence, by taking the limit as $k\to + \infty$ and using homogeneity property (\ref{Hhomog1}) we find that
\begin{multline} \label{formulone2}
\int_{\bhor{R} \setminus \Omega} 2v^{\gamma}S^2(W) = \\ =
	 -\gamma(\gamma-1)p(p-1) \int_{\bhor{R} \setminus \Omega}  v^{\gamma-2}V^2  - \gamma (2p-1) \int_{\bhor{R} \setminus \Omega} v^{\gamma-1} V \Delta_p^H v \\
  - \int_{\partial \Omega} \langle v^\gamma\sdueij{W}V_{\xi_i} + \gamma(p-1)v^{\gamma-1}V\nabla_\xi V ; \nu \rangle \\
  + \int_{\partial \bhor{R}} \langle v^\gamma\sdueij{W}V_{\xi_i}+ \gamma(p-1)v^{\gamma-1}V\nabla_\xi V ; \nu \rangle \,.
\end{multline}
Since from Theorem \ref{teostime} we have
$$
\lim_{R\to +\infty} \int_{\partial \bhor{R}} \langle v^\gamma\sdueij{W}V_{\xi_i}+ \gamma(p-1)v^{\gamma-1}V\nabla_\xi V ; \nu \rangle = 0
$$
and being $v$ the solution to \eqref{ext-pbv}, by taking the limit as $R\to +\infty$ in \eqref{formulone2} we obtain that
\begin{multline} \label{formulone3}
\int_{\rn \setminus \Omega} 2v^{\gamma}S^2(W) =
	 N(N-1)(p-1)^2 \int_{\rn \setminus \Omega}  v^{\gamma-2}V^2\\
  - \int_{\partial \Omega} \langle v^\gamma\sdueij{W}V_{\xi_i} + \gamma(p-1)v^{\gamma-1}V\nabla_\xi V ; \nu \rangle  \,,
\end{multline}
where we have used that $\gamma=1-N$.

Let us consider the surface integrals appearing in \eqref{formulone3}. By recalling that $\sdueij{W}=-w_{ji} + \delta_{ij} \tr(W)$, the fact that $\nu=-Dv/|Dv|$, and the homogeneity properties (\ref{Hhomog1}),(\ref{Hhomog2}), we obtain
	$$
	\int_{\bd\Omega} v^{\gamma}\langle \sdueij{W}V_{\xi_i};\nu \rangle =
	(p-1) \int_{\bd\Omega}v^{\gamma} \frac{H^{2p-2}}{|Dv|}H_{\xi_k}H_{\xi_i}v_{ki}-\int_{\bd\Omega}v^{\gamma}\Delta^H_pv\frac{H^p}{|Dv|} \,,
	$$
	and from (\ref{HDeltap=M}) we find 
	\begin{equation}\label{ints2ijV}
	\int_{\bd\Omega} v^{\gamma}\langle \sdueij{W}V_{\xi_i};\nu \rangle =-\int_{\bd\Omega}v^{\gamma}H^{2p-2}(Dv)H(\nu)\Huno_\Omega \,.
	\end{equation}
	Since
	\begin{equation*} 
	\int_{\bd\Omega} v^{\gamma-1}\langle{V}\nabla_\xi V;\nu\rangle = -\frac{1}p \int_{\bd\Omega}v^{\gamma}\frac{H^{2p-1}(Dv)}v H(\nu) \,,
	\end{equation*}
	from \eqref{formulone3} and \eqref{ints2ijV} we have
	\begin{eqnarray*}
		&&\int_{\rn\setminus\overline{\Omega}}2v^{\gamma}S^2(W)=\\\nonumber
		&&\quad -\int_{\bd\Omega} v^{\gamma}{H^{2p-2}(Dv)} H(\nu) \Big( \Huno(\bd\Omega) -\gamma\frac{p-1}p\frac{H(Dv)}v \Big)
		+N(N-1)({p-1})^2\int_{\rn\setminus\overline{\Omega}}v^{\gamma-2}V^{2},\nonumber
	\end{eqnarray*}
	which entails (\ref{S2W=equazione-esterno}) by recalling that $v$ solves Problem (\ref{ext-pbv}).
\end{proof}	

From Lemma \ref{lemma-newton} we obtain the following result.
	
\begin{cor}\label{ext-condiz-int}
	Let $\Omega$ and $H$ be as in Lemma \ref{lemma_ident_ext}. Then	
	\begin{equation}\label{BiCicondiz1}
	\int_{\bd\Omega}H(\nu)H^{2p-1}(Du) \left(  \frac{\Huno_\Omega}{N-1}-\frac{p-1}{p-N}\frac{H(Du)}u \right)  \; d\sigma\ge 0,	
	\end{equation}
	where $u$ is the solution of \eqref{ext-pbu}. Moreover, the equality sign holds if and only if the equality sign holds in (\ref{newtonIneq}) for the matrix $W=\nxi^2VD^2v$, being $v=u^\frac{p}{p-N}$ and $V$ as in (\ref{V_def}).
	\end{cor}
\begin{proof}
	Let us consider the function $v=u^{p/(p-N)}$ in $\rn\setminus\overline{\Omega}$.
	The proof follows by coupling Lemma \ref{lemma-newton} and (\ref{S2W=equazione-esterno}).
	Indeed, since $\tr(W)=\Delta^H_pv$ and  $v$ solves (\ref{ext-pbv}), by (\ref{newtonIneq}) it holds
	$$
	\int_{\rn\setminus\overline{\Omega}}2v^{\gamma}S^2(W)dx \le \frac{(p-1)^2}{p^2} N(N-1)\int_{\rn\setminus\overline{\Omega}}v^{\gamma-2}H^{2p}(Dv) dx \,.
	$$
	From (\ref{S2W=equazione-esterno}) we obtain
	$$
	-(N-1)\int_{\bd\Omega} {H^{2p-2}(Dv)} H(\nu) \Big( \frac{\Huno_\Omega}{N-1} -\frac{p-1}p\frac{H(Dv)}v \Big)\le 0,
	$$		
	which is equivalent to (\ref{BiCicondiz1}) for the function $u$ recalling that 
	$$
	H(Dv)=\frac p{(N-p)} u^{\frac{N}{p-N}}H(Du).
	$$
	Since the only involved inequality in this argument is \eqref{newtonIneq} applied to $W$, the characterization of the equality case follows.
\end{proof}

\section{Steps 2 and 3 - Proof of main Theorems}  \label{section_final_proofs}
In this section we complete the proof of the main theorems. More precisely, we tackle Steps 2 and 3 as outlined in Subsection \ref{subsect_ideas}.
Step 2 is achieved in Lemmas \ref{lemmaOmegaWulff-int} and \ref{lemmaOmegaWulff-ext} for the interior and exterior problems, respectively. Step 3 is carried out in the final part of the proof of Theorems \ref{thm_serrin_int}, \ref{thm_serrin_ext}, \ref{thm_agomaz}.

\subsection{Proof of  Theorem \ref{thm_serrin_int}} \label{subsect_int}

We start with the proof of Theorem \ref{thm_serrin_int}. In the next lemma we show that adding the overdetermined condition (\ref{HDuConst-int}) to Problem (\ref{int-pbu}) is equivalent
to imposing the condition that the equality sign holds for the matrix $W=\nxi^2V(Du) D^2u$ in \eqref{newtonIneq}. Later, we will use these conditions to obtain the Wulff shape characterization of $\Omega$.

\begin{lemma}\label{lemmaOmegaWulff-int}
Let $\Omega \subset \rn$ be a bounded domain with boundary of class $C^{2,\alpha}$ and let $H$ be a norm of $\rn$ in the class $\hnorm$, $p>1$. Let $u$ be the solution to Problem (\ref{int-pbu}).

The equality sign is attained in \eqref{newtonIneq} for the matrix $W=\nabla^2_\xi V(Du)D^2u$ if and only if $H(Du)$ is constant on $\bd\Omega$.
%
%
\end{lemma}

\begin{proof}
	Let us assume that the equality sign in (\ref{newtonIneq}) is attained by $W$.
	By the characterization of the equality case in Lemma \ref{lemma-newton}, for every $x\in\Omega$ there exists a function $\la(x)$ such that such that $W(x)=\la(x)\, Id$.
	Since $\tr(W)=\Delta^H_pu=-1$ in $\Omega$, then $\la(x)=-1/N$. Hence, $D^2u=-N^{-1} (\nabla_\xi^2V(Du))^{-1}$, that is from \eqref{eq_inversa_D2V_II}
	\begin{equation}\label{uij}
	u_{kj}(x)=-\frac{H^{-p}(Du(x))}{N(p-1)} \left(u_k(x) u_j(x)+(p-1)H^2(Du(x)) \partial_{\eta_k\eta_j}H_0(\nabla_\xi H(Du(x))) \right) \,.
	\end{equation}

	
\noindent Let $x \in \overline{\Omega}$ be such that $Du(x) \neq 0$. By using the expression of $u_{ij}$ in (\ref{uij}) we have
	\begin{equation*}
	\begin{aligned}
	\partial_{x_j}V(Du) & = H^{p-1}(Du) H_{\xi_k}(Du)u_{kj} \\
	&= -\frac{1}{N(p-1)} H^{-1}(Du)H_{\xi_k}(Du)\Big(u_ku_j+(p-1)H^2(Du)\partial_{\eta_k\eta_j}H_0(\nabla_\xi H)\Big) \\
	&= - \frac{u_j(x)}{N(p-1)} \,,
	\end{aligned}
	\end{equation*}
where the last equality follows from  the homogeneity properties (\ref{Hhomog1}) for $H$ and (\ref{Hhomog2}) for $H_0$. 
This implies that 
$$
V(Du(x)) = - \frac{u(x)}{N(p-1)} + d_i
$$
in any connected component $\Omega_i$ of $\overline{\Omega} \setminus \{Du = 0\}$ for some constants $d_i$. 

We notice that the set $\{Du=0\}$ is strictly contained in $\Omega$, since $H(Du)>0$ on $\partial \Omega$, and it has no interior points (this immediately follows by arguing by contradiction and testing $\Delta_p^H u=-1$ in $\Omega$ with a positive test function with support in $\{Du=0\}$).
Since $V(Du)$ is continuous in $\overline\Omega$, then the $d_i$'s coincide, i.e. 
$$
V(Du(x))=- \frac{u(x)}{N(p-1)} + d 
$$
in $\overline \Omega$ for some $d$. This implies that  $V(Du(x))$ is constant on $\partial \Omega$, and hence $H(Du)$ is constant on $\partial \Omega$.

Now, assume that $H(Du)=C$ on $\partial \Omega$. It is enough to prove that the equality sign holds in \eqref{ineq_interior}, i.e. that
$$
N(p-1) C^{2p-1} P_H(\Omega) = \big( N(p-1) + p \big) \int_\Omega u \, dx \,.
$$
We notice that by integrating $\Delta_p^H u = -1 $ in $\Omega$ and using that $H(Du)=C$ we find that 
$$
|\Omega|= C^{p-1} P_H(\Omega) \,,
$$
and from Lemma \ref{lemma-pohoz} we obtain
$$
\big( N(p-1) + p \big) \int_{\Omega}u\,dx=(p-1)C^p N |\Omega| = N(p-1) C^{2p-1} P_H(\Omega) \,,
$$
and we conclude.
\end{proof}

\begin{proof}[Proof of Theorem \ref{thm_serrin_int}]
Since $H(Du)=C$ on $\partial \Omega$, from Lemma \ref{lemmaOmegaWulff-int} we have that $W=\lambda (x) Id$. As done in the proof of Lemma \ref{lemmaOmegaWulff-int}, this implies that $\lambda$ is constant and $D^2 u = -N^{-1} (\nabla_\xi V(Du))^{-1}$ and then, from \eqref{eq_inversa_D2V_II}, $D^2u$ is given by \eqref{uij}. Being $\Delta_p^H u = -1$ and using \eqref{HDeltap=M} on $\partial \Omega$, we have
$$
\begin{aligned}
\Huno_{\partial \Omega} & = H^{1-p}(Du) \big(1 + (p-1) H^{p-2}(Du) H_{\xi_k}(Du)  H_{\xi_i}(Du) u_{ki} \big) \\
& = \frac{N-1}{N} H^{1-p}(Du) \,,
\end{aligned}
$$
where in the last equality we have used \eqref{uij}, \eqref{Hhomog1} and \eqref{Hhomog2}. Since $H(Du)=C$ on $\partial \Omega$ we have that 
 $\Huno_{\partial \Omega}$ is constant, and $\Omega$ is Wulff shape by Alexandrov's Theorem (see Theorem \ref{AleksandrovThm}). Formula \eqref{u=} follows at once.
\end{proof}

\subsection{Proof of Theorem \ref{thm_serrin_ext}} \label{subsect_ext}
We first show that adding the overdetermined condition (\ref{HDuConst-int}) to problem (\ref{int-pbu}) is equivalent
to imposing that the equality sign holds for the matrix $W=\nxi^2V(Dv) D^2v$ in Newton's inequality \eqref{newtonIneq}, which corresponds to Step 2 in the description in Subsection \ref{subsect_ideas}. Later, we will use these conditions to obtain the Wulff shape characterization of $\Omega$.

\begin{lemma}\label{lemmaOmegaWulff-ext}
	Let $\Omega$ be a convex domain with boundary of class $C^{2,\alpha}$, $H$ be a norm of $\rn$ in the class $\hnorm$, $p>1$, and $V$ be as in (\ref{V_def}). Let $u$ be the solution to Problem (\ref{ext-pbu}) and set $v=u^{\frac p{p-N}}$.
	The equality sign in (\ref{newtonIneq}) is attained by the matrix $W=\nabla^2_\xi V(Dv)D^2v$ if and only if $H(Du)$ is constant on $\bd\Omega$.
\end{lemma}
\begin{proof}
We first assume that $W$ realizes the equality (\ref{newtonIneq}) and we want to prove that $H(Du)$ is constant on $\partial \Omega$. The characterization of the equality case in Lemma \ref{lemma-newton} implies that there exists a function $\lambda(x)$ such that $W(x)=\lambda(x) Id$ for every $x\in\rn\setminus\overline{\Omega}$.
Since $W(x)=D_x \nxi V(Du(x))$ then the mean value theorem implies that
$$
V_{\xi_i}(Du(x))=f_i(x_i) \,,
$$
for some functions $f_i:\RR \to \RR$, $i=1,\ldots,N$. Hence $\lam=f_i'(x_i)$ for any $i=1,\ldots,N$ and then $\lambda$ is constant. Thus 
$$
\nxi^2V(Dv) D^2v=\lambda Id \,,
$$
for some constant $\lambda$, i.e. $D^2v= \lambda (\nxi^2V(Dv))^{-1}$ and \eqref{eq_inversa_D2V_II} yields
\begin{equation}\label{vij}
	v_{ij}=\frac{\lambda}{p-1}   H^{-p}(Dv)\left( v_iv_j +(p-1)H^2(Dv)\partial^2_{\eta_i\eta_j} H_0(\nabla_\xi H(Dv)) \right) \,.
\end{equation}
Moreover, since $W=\lambda Id$ and $\Delta^H_pv=\tr(W)=N\lambda$, from \eqref{ext-pbv} we have
$$
\lambda = \frac{p-1}{p} \frac{H^p(Dv)}{v},
$$
in $\rn\setminus\overline{\Omega}$ which implies that $H(Dv)$ is constant on every level line of $v$. In particular, we obtain that 
$H^p(Dv)=p\lambda/(p-1)$ on the boundary of $\Omega$, which implies the conclusion.
	
Now we prove the reverse assertion. Assume that $H(Du)$ is constant and equal to $C$ on $\bd\Omega$. By Corollary \ref{ext-condiz-int} we have
	$$
   	\int_{\bd\Omega}H(\nu) \left(  \frac{\Huno_\Omega}{N-1}-\frac{p-1}{p-N}{C} \right) \ge 0.	
	$$
	We recall that the value of $C$ is given by (\ref{valueC}) and hence we obtain
	$$
	\int_{\bd\Omega}H(\nu)\frac{\Huno_\Omega}{N-1}\,\ge \frac{p-1}{p-N}C\; P_H(\Omega)= \frac 1N \frac{P_H^2(\Omega)}{|\Omega|}.
	$$
	Moreover, thanks to the anisotropic Minkowski inequality (see \cite{BCS}, Proposition 2.9), it holds
	$$
	\int_{\bd\Omega}H(\nu)\frac{\Huno_\Omega}{N-1}\, d\sigma(x) \le  \frac 1N \frac{P_H^2(\Omega)}{|\Omega|},
	$$
	and hence equality holds in (\ref{BiCicondiz1}), which implies that $W=\nabla^2_\xi V(Dv)\, D^2v$ achieves the equality sign in \eqref{newtonIneq}.
\end{proof}

We are now able to give the proof of Theorem \ref{thm_serrin_ext}.

\begin{proof}[Proof of Theorem \ref{thm_serrin_ext}.]
The proof is analogous to the one of Theorem \ref{thm_serrin_int}. Indeed, Lemma \ref{lemmaOmegaWulff-ext} implies that $W$ is a multiple of the identity matrix and from \eqref{eq_inversa_D2V_II} we obtain an explicit expression for $D^2v$ (see \eqref{vij}). Since $v$ satisfies \eqref{ext-pbv} and \eqref{HDvcost}, then $\Delta_p^H v$ is constant on $\partial \Omega$ and from \eqref{vij} and \eqref{HDvcost} we obtain that the mean curvature $\Huno_{\partial \Omega}$ of $\partial \Omega$ is constant. Hence, $\Omega$ is Wulff shape from Alexandrov's Theorem \ref{AleksandrovThm}. The explicit expression of $u$ \eqref{uexplext} follows easily.	
\end{proof}

%


\subsection{Proof of Theorem \ref{thm_agomaz}}
As a byproduct of our technique we are able to prove another Wulff characterization in terms of solution to \eqref{ext-pbu} when the integral overdetermined condition (\ref{agomazz-u}) is considered. More precisely we can prove Theorem \ref{thm_agomaz}.

\begin{proof}[Proof of Theorem \ref{thm_agomaz}.]
Let $u$ be the solution to (\ref{ext-pbu}) and let $v=u^\frac{p}{p-N}$.
Thanks to Lemma \ref{ext-condiz-int}, if condition (\ref{agomazz-u}) holds, then equality holds in (\ref{BiCicondiz1}) and hence, following the proof of Corollary \ref{ext-condiz-int}, equality holds in (\ref{newtonIneq}) for the matrix $W=\nabla^2_\xi V (Dv) D^2v$.
Lemma \ref{lemmaOmegaWulff-ext} guarantees that $H(Du)$ is constant on $\bd\Omega$ and hence $u$ solves (\ref{ext-pbu}),(\ref{HDuConst-ext}) and the Wulff shape of $\Omega$ follows by Theorem \ref{thm_serrin_ext}.
\end{proof}


\appendix

\section{The constant $C$ in the exterior problem} \label{appendix-valoreC}
In this section we compute the value of the constant $C$ which appears in the Neumann boundary constraint \eqref{HDuConst-ext}. More precisely we show that, in order to have existence of a solution to the overdetermined problem \eqref{ext-pbu}--\eqref{HDuConst-ext}, an apriori relation between the value of $C$ and the geometry of the set $\Omega$ must hold.

\begin{prop}
	If there exists a solution $u\in C^{2,\alpha}(\rn\setminus\Omega)$ to Problem \eqref{ext-pbu}--\eqref{HDuConst-ext}, then
$$
C=\frac{N-p}{N\,(p-1)}\frac{P_H(\Omega)}{|\Omega|}.
$$
\end{prop}
\begin{proof}
We split the proof in two main steps in which we compute the anisotropic $p$-capacity, the anisotropic perimeter and the volume of the set $\Omega$	and related quantities.
%
%

We denote by $\nu_t$ the inner unit normal vector to $D_t=\{u>t\}$, and by $\nu=\nu_1$ the inner unit normal to $\Omega$, so that $\nu_t =Du/|Du|$ on $\bd D_t$.

\vspace{0.5em}

\noindent\emph{First step:} $C=\frac{P_H(\Omega)}{\pCap(\Omega)}\frac{C^p}p.$

Let us remark that $\int_{\bd D_t} H^{p-1}(Du)\langle \nxi H(Du);\nu_t\rangle\,d\mathcal{H}^{N-1}$ is independent of $0<t\le 1$: indeed since $\Delta_p^H u=0$ in $\rn \setminus \Omega $, the Divergence Theorem  entails
$$
\int_{\bd\Omega} H^{p-1}(Du)\,\langle \nxi H(Du);\nu\rangle\,d\mathcal{H}^{N-1} - \int_{\bd D_t} H^{p-1}(Du)\langle \nxi H(Du);\nu_t\rangle\,dx=0 \,,
$$
for any $0<t < 1$. Moreover, since $\nu_t=Du/|Du|$ and by the homogeneity property (\ref{Hhomog1}), we obtain
\begin{equation}\label{intDt}
\int_{\bd D_t} H^{p-1}(Du)\langle \nxi H(Du);\nu_t\rangle\,dx=\int_{\bd D_t}H^p(Du)\frac 1{|Du|}\,d\mathcal{H}^{N-1}\,,
\end{equation}
for any $0<t < 1$ and we find
\begin{equation} \label{alphha}
\int_\Omega H^p(Du)\, \frac{1}{|Du|} \,dx= \int_{\bd D_t} H^p(Du) \frac{1}{|Du|}\,d\mathcal{H}^{N-1} \,,
\end{equation}
for any $0<t < 1$. 

Now we compute $\pCap(\Omega)$. By using the definition (\ref{pcap}) and the coarea formula we have
$$
p\; \pCap(\Omega) = \int_{\rn\setminus\Omega}H^p(Du)\,dx=\int_0^1\int_{\bd D_t}H^p(Du)\frac 1{|Du|}\; d\haus\; dt \,.
$$
From \eqref{alphha} we have
$$
p\; \pCap(\Omega) =\int_{\bd\Omega} H^p(Du)\frac 1{|Du|}.
$$
By recalling the boundary condition (\ref{HDuConst-ext}) and again the homogeneity property (\ref{Hhomog1}), together with the definition of anisotropic perimeter (\ref{anis_surf_energy}), we conclude the first step since from the previous relation it follows
$$
p\; \pCap(\Omega)=C^{p-1}\int_{\bd\Omega} H(\nu) = C^{p-1} P_H(\Omega).
$$

\noindent\emph{Second step:} $\frac{C^p}p= \frac{(N-p)}{N(1-p)} \frac{\pCap(\Omega)}{|\Omega|}.$
Since $H(Du)=C$ on $\partial \Omega$, from the Divergence Theorem, Theorem \ref{teostime} and \eqref{pcap} we have
\begin{eqnarray}\label{CpNOmega} \nonumber
C^p N|\Omega| &=& -\int_{\bd\Omega} H^p(Du)\langle x;\nu \rangle = -\int_{\rn\setminus\Omega} \diver(x\, H^p(Du(x)))\\
&=& -Np\pCap(\Omega) -p\int_{\rn\setminus\Omega} H^{p-1}(Du) H_{\xi_k}(Du) \, u_{ki}x_i.
\end{eqnarray}

Notice that, by (\ref{Hhomog1}), we have $ H_{\xi_k}(Du) \, u_{ki}x_i = H_{\xi_k}(Du)(u_{ki}x_i +u_k) -H(Du)$ and hence
\begin{equation}\label{int1}
\int_{\rn\setminus\Omega}H^{p-1}(Du) H_{\xi_k}u_{ki}x_i = -p\pCap(\Omega)+\int_{\rn\setminus\Omega} H^{p-1}(Du)H_{\xi_k}(Du)\Big(u_i x_i \Big)_{k}.
\end{equation}
Moreover, since $\Delta_p^H u=0$ in $\rn\setminus\overline{\Omega}$, we have
$$
H^{p-1}(Du)H_{\xi_k}(Du)\Big(u_i x_i \Big)_{k}=\frac 1p \diver(\langle Du; x \rangle \; \nxi H^p(Du))
$$
and the Divergence Theorem yields
$$
\int_{\rn\setminus\Omega} H^{p-1}(Du)H_{\xi_k}(Du)\Big(u_i x_i \Big)_{k} = \int_{\bd\Omega} H^{p-1}(Du)\langle Du;x \rangle \langle \nxi H(Du); \nu \rangle.
$$
Recalling the homogeneity of $H$ (\ref{Hhomog1}), the fact that $\nu=Du/|Du|$ and the boundary condition (\ref{HDuConst-ext}), the previous equality entails
$$
\int_{\rn\setminus\Omega} H^{p-1}(Du)H_{\xi_k}(Du)\Big(u_i x_i \Big)_{k} =\int_{\bd\Omega} H^{p-1}(Du)\langle \nu;x \rangle \langle \nxi H(Du); Du \rangle=C^p \int_{\bd\Omega}\langle \nu;x \rangle = -C^p N |\Omega|,
$$
and hence coupling with (\ref{CpNOmega}) and (\ref{int1}) we obtain
$$
C^p N|\Omega|= -Np\pCap(\Omega) -p(-p\pCap(\Omega)-C^p N|\Omega|).
$$
This conclude the second step.

\noindent\emph{Conclusion:} From the previous two steps we immediately obtain that
$$
C=\frac{P_H(\Omega)}{|\Omega|}\frac{(N-p)}{N(1-p)} \,,
$$
and the proof is complete.
\end{proof}

\section{A lower bound on the gradient} \label{appendix_lower_bound}

\begin{lemma}\label{lemmaLewis}
Let $u_R$ be the solution to 
\begin{equation}\label{capBR_app}
\begin{cases}
\Delta^H_p u_R=0 &\qquad\text{ in }\bhor{R}\setminus\overline{\Omega} \,,\\
u_R=1 &\qquad\text{ on }\partial\Omega \,, \\
u_R=0 &\qquad\text{ on } \partial\bhor{R} \,.
\end{cases}
\end{equation} 
There exists a constant $C$, not depending on $R$ such that
$$
H(Du(x))\ge C H_0^{\frac{1-N}{p-1}}(x),
$$
for $x\in \overline{\bhor{R}}\setminus\Omega$.
\end{lemma}

\begin{proof}
We closely follow the proof of Lemma 2 in \cite{Le}, which needs to be adapted to the anisotropic setting.

To simplify the notation we set $E=\bhor{R}\setminus\overline{\Omega}$. Since $\partial E = \partial \bhor{R} \cup \partial \Omega$ is of class $C^{2,\alpha}$ and $H$ is uniformly convex, $E$ satisfies an interior touching ball condition of a radius $\delta$, i.e. for any $x_0 \in \partial E$ there exists a ball $\bhorx{\delta}{z} \subset E$ with $x_0 \in \partial \bhorx{\delta}{z}$.

We first prove that there exist $\lam_0>1$ and $c>0$ such that for any $y \in \partial E$ we have that 
\begin{equation} \label{lewis_1}
\begin{split}
 & 1-u(\lam y) \geq \delta^{\frac{1-N}{p-1}}c (\lam -1) \quad \text{if } y \in \partial \Omega \,, \\
 & u(y/\lam) \geq \delta^{\frac{1-N}{p-1}}c(\lam-1) \quad \text{if } y \in \partial \bhor{R} \,,
\end{split}
\end{equation} 
for any $1<\lam \leq \lam_0$. 

We prove \eqref{lewis_1} by using a barrier type argument. Let $x_0 \in \partial E$ and we choose $z$ as above, i.e. so that $x_0 \in \partial \bhorx{\delta}{z}$. Let 
$$
v(x)=
\begin{cases}
0 & \qquad\text{ in }\rn\setminus \bhorx{\delta}{z} \,, \\
1 & \qquad\text{ in }\bhorx{\delta/2}{z}\,,\\
\dfrac{H_0^{\pesp}(x-z)-\delta^{\pesp}}{\delta^{\pesp}(2^{\frac{N-p}{p-1}}-1)} & \qquad\text{ in }\bhorx{\delta}{z}\setminus\overline{\bhorx{\delta/2}{z}} \,,
\end{cases}
$$
and we notice that $v\in W_0^{1,p}(\rn)$, $\Delta_p^H v = 0$ in $\bhorx{\delta}{z}\setminus\overline{\bhorx{\delta/2}{z}}$ and, 
from (\ref{HDHo=1}), we have that $H(Dv(x))\ge c_1 \delta^{\frac{1-N}{p-1}}$ in $\bhorx{\delta}{z}\setminus\overline{\bhorx{\delta/2}{z}}$, where $c_1$ depends only on $N,p$. 
Now assume that $x_0 \in \partial \Omega$. From the convexity of $\Omega$, the strict convexity of $\bhor{R}$, there exist $\la_0,\mu$, depending only on $\Omega,N,p$ such that $\lam x_0 \in \bhorx{\delta}{z}\setminus\overline{\bhorx{\delta/2}{z}}$ and
$$
v(\la x_0)\ge \delta^{\frac{1-N}{p-1}} \mu (\la-1)
$$ 
for $1<\la\le\la_0$. Analogously, we can prove that if $x_0 \in \bhor{R}$, then $\lam^{-1} x_0 \in  \bhorx{\delta}{z}\setminus\overline{\bhorx{\delta/2}{z}}$ and
$$
v(x_0/\lam)\ge \delta^{\frac{1-N}{p-1}}\mu (\la-1)
$$ 
for $1<\la\le\la_0$. Notice that $\la_0$ can be chosen uniformly for any $x_0 \in \partial \Omega$. 

Let $G$ be such that $\bhor{R}\setminus\overline{\Omega}= G + \bhor{\delta/2}$ (here we mean the Minkowski sum). Since $0<u_R<1$ in $\bhor{R}\setminus\overline{\Omega}$, from Harnack's inequality (see \cite{Se_acta}) we have that 
$$
\min_G \{1-u_R,\,u_R\} \geq A \,,
$$
where $A>0$ depends only on $G$, $p$ and $n$. From the weak comparison principle we have that $1-u \geq A v$ in $\bhorx{\delta}{z}\setminus\overline{\bhorx{\delta/2}{z}}$ if $x_0 \in \partial \Omega$ and $u \geq Av$ in $\bhorx{\delta}{z}\setminus\overline{\bhorx{\delta/2}{z}}$ if $x_0 \in \partial \bhor{R}$, and hence \eqref{lewis_1} is proved.  

For $1<\lam \leq \lam_0$ we define $E_\lam=\bhor{R/\lam}\setminus\overline{\Omega}$ and we notice that the function $u_R(\lam x)$ satisfies $\Delta_p^H u_R(\lam x) =0$ in $E_\lam$. From \eqref{lewis_1}, the weak comparison principle implies that 
$$
u_R(\lam x) \leq u_R - \delta^{\frac{1-N}{p-1}} A \mu (\lam-1)
$$
in $E_\lam$ for $1<\lam \leq \lam_0$. Since 
$$
\delta^{\frac{1-N}{p-1}} A \mu \leq \lim_{\lam \to 1} \frac{u_R(x) - u_R(\lam x)}{\lam -1} = -Du(x) \cdot x \,,
$$
for any $x \in E$, Cauchy-Schwarz inequality yields 
$$
H(Du) \geq  A \mu \frac{\delta^{\frac{1-N}{p-1}}}{H_0(x)} \,.
$$
Being $\delta \leq H_0(x)$ in $E$, the conclusion follows from \eqref{norms equiv}.
\end{proof}

\end{document}